\documentclass[11pt]{amsart}
\usepackage{amssymb}
\usepackage{ bbold }
\usepackage{xypic}
\xyoption{all}
\usepackage{url}
\usepackage{mathrsfs}
\usepackage{tikz-cd}

\usepackage[linkcolor=blue, colorlinks=true]{hyperref}
\usepackage[mathscr]{euscript}

\title[On The Picard group of the stable module category]{On the Picard group of the stable module category for infinite groups}

\author[Juan Omar G\'omez]{Juan Omar G\'omez}
\thanks{}
\address{
Fakultät für Mathematik, Universität Bielefeld, D-33501 Bielefeld, Germany}
\email{jgomez@math.uni-bielefeld.de}

\newcommand{\comments}[1]{}

\newcommand{\Coker}{\operatorname{Coker}\nolimits}

\newcommand{\Ho}{\operatorname{Ho}\nolimits}
\newcommand{\Hom}{\operatorname{Hom}\nolimits}

\newcommand{\Ker}{\operatorname{Ker}\nolimits}

\newcommand{\Mod}{\operatorname{Mod}\nolimits}
\newcommand{\Ext}{\operatorname{Ext}\nolimits}

\newcommand{\Res}{\operatorname{Res}\nolimits}

\newcommand{\StMod}{\operatorname{StMod}\nolimits}
\newcommand{\CAlg}{\operatorname{CAlg}\nolimits}

\newcommand{\ind}{\mathord\uparrow}
\newcommand{\restr}{\mathord\downarrow}

\def \C{{\mathcal C}}
\def \D{{\mathcal D}}

\def \gH{ { ^g H} }

\makeatletter
\newcommand*{\doublerightarrow}[2]{\mathrel{
  \settowidth{\@tempdima}{$\scriptstyle#1$}
  \settowidth{\@tempdimb}{$\scriptstyle#2$}
  \ifdim\@tempdimb>\@tempdima \@tempdima=\@tempdimb\fi
  \mathop{\vcenter{
    \offinterlineskip\ialign{\hbox to\dimexpr\@tempdima+1em{##}\cr
    \rightarrowfill\cr\noalign{\kern.5ex}
    \rightarrowfill\cr}}}\limits^{\!#1}_{\!#2}}}
\newcommand*{\triplerightarrow}[1]{\mathrel{
  \settowidth{\@tempdima}{$\scriptstyle#1$}
  \mathop{\vcenter{
    \offinterlineskip\ialign{\hbox to\dimexpr\@tempdima+1em{##}\cr
    \rightarrowfill\cr\noalign{\kern.5ex}
    \rightarrowfill\cr\noalign{\kern.5ex}
    \rightarrowfill\cr}}}\limits^{\!#1}}}
\makeatother

\theoremstyle{plain}
\newtheorem*{introtheorem}{Theorem}

\newtheorem{theorem}{Theorem}[section]
\newtheorem{proposition}[theorem]{Proposition}
\newtheorem{corollary}[theorem]{Corollary}
\newtheorem{lemma}[theorem]{Lemma}

\theoremstyle{definition}
\newtheorem{definition}[theorem]{Definition}
\newtheorem{remark}[theorem]{Remark}
\newtheorem{example}[theorem]{Example}

\date{\today}

\begin{document}

\begin{abstract}
We introduce the stable module $\infty$-category for groups of type $\Phi$ as an enhancement of the stable category defined by N. Mazza and P. Symonds. For groups of type $\Phi$ which act on a tree, we show that the stable module $\infty$-category decomposes in terms of the associated graph of groups. For groups which admit a finite-dimensional cocompact model for the classifying space for proper actions, we exhibit a decomposition in terms of the stable module $\infty$-categories of their finite subgroups. We use these decompositions to provide methods to compute the Picard group of the stable module category. In particular, we provide a description of the Picard group for countable locally finite $p$-groups. 
\end{abstract}

\maketitle

\section*{Introduction}
The stable module category has been proven widely useful to deal with classification problems in modular representation theory of finite groups. It has been studied from many different viewpoints, including homotopy theory that had a successful application in classification of endotrivial modules for finite groups \cite{Gro22}.  

For infinite groups much less is known. Mazza and Symonds construct the stable module category for groups of type $\Phi$, as the largest quotient of the category of modules on which the syzygy functor is invertible \cite{MS}. This category is equipped with a triangulated structure compatible with the tensor product over the ground ring. In other words, it is a tensor triangulated category in the language of Balmer \cite{Bal10}. An important invariant of this category is its Picard group, which for finite groups agrees with the group of endotrivial modules. However, for infinite groups we have just a few available tools to compute the Picard group of the stable module category. 


We introduce a homotopy-theoretic interpretation of the stable module category for groups of type $\Phi$ as a symmetric monoidal stable $\infty$-category. For finite groups, this interpretation agrees with the one given in \cite{Mat15}. In particular, for groups of type $\Phi$ acting on a tree, we exhibit a decomposition in terms of the fundamental domain of the action and its isotropy groups. No restriction on the size of the isotropy groups is needed for the following result (see Theorem \ref{decomposition of C}).

\begin{introtheorem}
    Let $G$ be a group of type $\Phi$ acting on a tree. Consider the associated graph of groups 
    $\Gamma(G)\colon\Gamma \to \mathrm{Gps}$, that is, $G$ is the fundamental group of $\Gamma(G)$. Then we have an equivalence of symmetric monoidal $\infty$-categories 
$$\StMod(kG)\xrightarrow[]{\simeq} \varprojlim_{\sigma\in \Gamma^\textrm{op}} \StMod(kG_\sigma). $$ 
    
\end{introtheorem}

This decomposition leads to a spectral sequence which computes the Picard group of the stable module category for groups of type $\Phi$ acting on trees. In particular, we provide a more conceptual proof of Theorem 7.1 in \cite{MS}. That is, we obtain a short exact sequence of abelian groups (see Corollary \ref{exact sequence for picard groups})
$$0\to H^1(\Gamma; \pi_1\circ f)\to T(G)\to H^0(\Gamma; \pi_0\circ f)\to 0$$ where $T(G)$ denotes the Picard group of the stable module category of $G$, and $f$ is the composition of \textit{the Picard space} functor and the stable module $\infty$-category functor.

For groups admitting a finite-dimensional model for the classifying space for proper actions, we follow the ideas in \cite{Mat16} to exhibit a different decomposition of the stable module $\infty$-category, in this case, in terms of the finite subgroups. The following theorem summarizes the results in Subsection  \ref{subsection 4.2}.

\begin{introtheorem}
     Let $G$ be a group admitting a finite-dimensional cocompact model $X$ for $\underline{E}G$. Let $\mathscr{F}$ be a family of finite subgroups of $G$ which contains the family of finite $p$-subgroups of $G$ and $\textrm{Stab}$ be the family of finite subgroups of $G$ of the form $G_\sigma$, for some simplex $\sigma$ of  $X$. Then we have equivalences of symmetric monoidal stable $\infty$-categories $$\StMod(kG)\xrightarrow[]{\simeq} \varprojlim_{G/H\in\mathcal{O}_\mathscr{F}(G)^\textrm{op}} \StMod(kH) $$  
     and $$\StMod(kG)\xrightarrow[]{\simeq} \varprojlim_{G/H\in\mathcal{O}_\textrm{Stab}(G)^\textrm{op}} \StMod(kH). $$
\end{introtheorem}

We provide computations for certain classes of groups. For instance, we use the classification of endotrivial modules for finite $p$-groups \cite{CT04},\cite{CT05} and Corollary \ref{exact sequence for picard groups} to determine (as an abstract group) the Picard group of the stable module category for countable locally finite $p$-groups. The following theorem summarizes the results of Subsection \ref{Computations locally finite p-groups}.

\begin{introtheorem} Let $P$ be a countable locally finite $p$-group. Then the following hold. 

\begin{itemize}
    \item[(a)] If $P=\mathbb{Z}/p^\infty$, then $T(P)\cong \mathbb{Z}/2$. 
    \item[(b)] Let $D_{2^{\infty}}=\bigcup D_{2^n}$, where $D_{2^n}$ denotes the dihedral group of order $2^n$. Then $T(D_{2^\infty})\cong \mathbb{Z}$.
    \item[(c)] Let $Q_{2^\infty}=\bigcup Q_{2^n}$, where $Q_{2^n}$ denotes the generalized quaternion group of order $2^n$. Then $T(Q_{2^\infty})\cong \mathbb{Z}/4$.  
    
    \item[(d)]  Suppose that $P$ is artinian and that it admits a tower  $$Q_1\leq Q_2\leq ...$$ whose union is $P$ and such that $Q_n$ is not  cyclic, dihedral, semi-dihedral or quaternion for all $n\geq1$. Then

$$
T(P) = \left\{
        \begin{array}{ll}
           
            \mathbb{Z}^r & \textrm{ if } P \textrm{ has }p\textrm{-rank  at most 2}\\
            \mathbb{Z}^{r+1} & \textrm{ if } P \textrm{ has }p\textrm{-rank  at least 3}
        \end{array}
    \right.
$$

\noindent where $r$ is the number of conjugacy classes of maximal elementary abelian subgroups of $P$ of rank 2.  

\item[(e)] If $P$ is not artinian, then $T(P)\cong\mathbb{Z}$.
\end{itemize}
\end{introtheorem}

Finally, we provide a tool to compute the Picard group for a certain class of groups of type $\Phi$ that we call amalgam groups. These groups have geometric dimension 2 with respect to the family of finite groups. Let $G$ be an amalgam group acting on a tree. Then using Corollary \ref{exact sequence for picard groups} to compute the Picard group of $G$ will involve computing invariants of the stable module $\infty$-category for infinite groups, which could be as hard to compute as the invariants of the stable module $\infty$-category of $G$. Hence the importance of the following result (see Corollary \ref{exact sequence for triangles of groups}). 

\begin{introtheorem}
     Let $G$ be a group admitting a 2-dimensional model $X$ for $\underline{E}G$  such that the fundamental domain of the action is homeomorphic to the standard 2-simplex $\Delta^2$. Let $\mathcal{T}$ denote the barycentric subdivision of $\Delta^2$. If the isotropy group of the 2-simplex is trivial,  then we have an exact sequence of abelian groups $$0\to H^1(\mathcal{T}; \pi_1\circ f)\to T(G)\to H^0(\mathcal{T}; \pi_0\circ f)\to0$$ where $f$ is the composition of the Picard space functor and the stable module $\infty$-category functor. 
\end{introtheorem}


\noindent \begin{bf}Acknowledgments.\end{bf} I heartily thank my supervisor Jos\'e Cantarero for his support and encouragement. I would like to thank Omar Antol\'in and Anish Chedalavada for helpful conversations on this work. I am very grateful to Peter Symonds for his hospitality during a visit to the University of Manchester, and for several interesting discussions on this project. I thank Luis-Jorge Sánchez-Saldaña and Ian Leary for useful remarks on the poset $\mathscr{S}_H$ appearing in the proof of Proposition \ref{cofinalidad de la inclusion de subgrupos de isotropia}. I also thank two anonymous referees for their valuable comments, especially for pointing out a mistake in an early version of this work.  
I further wish to thank the Hausdorff Research Institute for Mathematics for the hospitality during the
trimester program ‘Spectral Methods in Algebra, Geometry, and Topology’, funded
by the Deutsche Forschungsgemeinschaft under Germany’s Excellence Strategy–EXC-2047/1–390685813. 

This work was supported by the \textit{Consejo Nacional de Humanidades, Ciencias y Tecnologías} under the program ‘Becas Nacionales de Posgrado’, and by the  \textit{Deutsche Forschungsgemeinschaft} (DFG, German Research Foundation) — SFB-TRR 358/1 2023 — 491392403.

\section{Preliminaries}

In this paper, we let $k$ denote a field of prime characteristic $p$. For a given group $G$, we let $\Mod(kG)$ denote the category of all $kG$--modules, and let $\otimes$ denote $\otimes_k$ equipped with the diagonal action of $G$, unless it is specified otherwise.  
 

\subsection{Groups of type $\Phi$}

Let $\textrm{projdim}_{kG} M$ denote the projective dimension of the $kG$--module $M$, that is, the shortest possible length of a projective resolution of the module or $\infty$ if there is no finite resolution. 

\begin{definition}
Let $G$ be a group. If any $kG$--module $M$ satisfies that $\textrm{projdim}_{kG} M$ is finite if and only if $\textrm{projdim}_{kF} M$ is finite for any finite subgroup $F$ of $G$, then we say that $G$ is a group of \textit{type $\Phi$} (or $\Phi_k$ if we need to emphasise the role of $k$). 
\end{definition}

The following result gives us a large collection of groups of type $\Phi$ \cite[Proposition 2.5]{MS}. In particular, groups of finite virtual cohomological dimension are groups of type $\Phi$, as well as groups that admit a finite-dimensional model for the classifying space for proper actions.

\begin{proposition}\label{exact sequence for groups of type}
Let $G$ be a group. If there exists an exact complex of $kG$--modules $$0\to C_n\to...\to C_0\to k\to 0$$ such that each $C_i$ is a direct summand of a sum of modules of the form $k\ind_H^G$ with $H$ of type $\Phi$ and $\textrm{findim} (kH)\leq m$ for a fixed $m$, then $G$ is of type $\Phi$. 
\end{proposition}

Moreover, if $G$ is a group of type $\Phi$, then the  \textit{finitistic dimension} of $kG$ is finite. Recall that the finitistic dimension is given by $$\sup\{\textrm{projdim}_{kG} M \ | \ \textrm{projdim}_{kG}M<\infty\}.$$ In particular, a free abelian group of infinite rank\footnote{The rank of an abelian group is just the rank as a $\mathbb{Z}$--module.} is not of type $\Phi$.


\subsection{The stable module category for groups of type $\Phi$}\label{stable module category}

Let $G$ be a group. The \textit{projective stable module category} $\underline{\Mod}(kG)$ is the category whose objects are $kG$--modules and morphisms are equivalence classes of homomorphisms   under the relation $f\sim g$ if and only if $f-g$ factors through a projective. We have a self-functor $\Omega\colon \underline{\Mod}(kG)\to\underline{\Mod}(kG)$ given by taking the kernel of a surjection $P\to M$, where $P$ is projective. In particular,  for any pair $M,N$  of $kG$--modules there is a natural map $$\Omega\colon\underline{\Hom}_{kG}(M,N)\to \underline{\Hom}_{kG}(\Omega M,\Omega N).$$

\begin{definition}
    Let $M,N$ be $kG$--modules. The \textit{complete cohomology} is defined  via $$\widehat{\Ext}^r_{kG}(M,N):=\lim_{n}\Hom_{\underline{\Mod}(kG)}(\Omega^{n+r}M,\Omega^nN).$$
\end{definition}

In particular, a $kG$--module $M$ has finite projective dimension if and only if $\widehat{\Ext}^0_{kG}(M,M)=0$ (see \cite[ Lemma 3.1]{Ben97}). 

\begin{definition}\label{stable module category M-P}
Let $\underline{\StMod}(kG)$ be the category whose objects are all the $kG$--modules and the morphisms between two objects $M, N$ are given by complete cohomology $$\Hom_{\underline{\StMod}(kG)}(M,N)=\widehat{\Ext}^0_{kG}(M,N).$$
\end{definition}

\begin{remark}
For groups of type $\Phi$, let $\widehat{H}^i(G;N)$ denote $\widehat{\Ext}_{kG}^i(k,N)$ for all $i\in \mathbb{Z}$. This coincides with generalized Tate-Farrell cohomology defined by Ikenaga \cite{Ike} when the group has finite generalized cohomological dimension. In particular, for finite groups it coincides with Tate cohomology. 
\end{remark}


Let $\textrm{K}_{tac}(kG)$ denote the category of totally acyclic complexes of projective $kG$--modules and chain maps up to homotopy. Let $\textrm{GP}(kG)$ (resp. $\underline{\textrm{GP}}(kG)$) denote the full subcategory of $\Mod(kG)$ (resp. $\underline{\Mod}(kG)$)  on the \textit{Gorenstein projective modules}, i.e. the modules that are isomorphic to a kernel in a totally acyclic complex of projectives.

 For groups of type $\Phi$, any complex with finitely many non-zero homology groups has a complete resolution. Then we have a functor  $$\textrm{CompRes}\colon D^b(\Mod(kG)) \to K_{tac}(kG)$$ where $D^b(\Mod(kG))$ is the derived  category of complexes of $kG$--modules with only finitely many nonzero homology groups. The kernel of  this functor contains the homotopy category of bounded complexes of projective $kG$--modules denoted by $K^b(\mathrm{Proj}(kG))$.

The categories $\underline{\mathrm{GP}}(kG),\ D^b(\Mod(kG))/K^b(\textrm{Proj}(kG))$ and $K_{\textrm{tac}}(kG)$ are equivalent as triangulated categories (this holds for more general rings, see \cite{Bel}, \cite{Buc87}). Mazza and Symonds prove that these equivalences factor through $\underline{\StMod}(kG)$ for groups of type $\Phi$ (see \cite[Theorem 3.9]{MS}).  

\begin{theorem}

For groups of type $\Phi$, the following categories are equivalent.

\begin{itemize}
    \item $\underline{\StMod}(kG)$.
    \item $D^b(\Mod(kG))/K^b(\mathrm{Proj}(kG))$.
    \item $K_{tac}(kG)$.
    \item $\underline{\textrm{GP}}(kG)$.
\end{itemize}

\noindent They are equivalent as triangulated categories, except for $\underline{\StMod}(kG)$. 

\end{theorem}

These equivalences define a triangulated structure on $\underline{\StMod}(kG)$. The
distinguished triangles are all the triangles isomorphic to a short exact sequence of modules and the shift $\Omega^{-1}$ is obtained by taking the kernel in degree -1 of a complete resolution of the module.


\subsection{The Picard group of $\underline{\StMod}(kG)$}

Recall that the Picard Group $\textrm{PicGp}(\C)$ of a symmetric monoidal category $(\C,\otimes,\mathbf{1})$ is  the group of isomorphism classes of objects which are invertible with respect to the tensor product $\otimes$. 

We will consider $\underline{\StMod}(kG)$ with the symmetric monoidal structure inherited from the symmetric monoidal structure of $\Mod(kG)$ given by 
the tensor product over $k$. In fact, this structure turns $\underline{\StMod}(kG)$ into a tensor triangulated category in the sense of Balmer. In particular, we will write $T(G)$ to denote the Picard group of $\underline{\StMod}(kG)$. Note that if $G$ is finite, this group corresponds to the so-called group of endotrivial modules (see \cite{Maz}, \cite{The}). In this context, we say that a $kG$--module $M$ is \textit{invertible} if $[M]$ belongs to $T(G)$.

\begin{proposition}
Let $G$ be a group of type $\Phi$. Let $M$ be a $kG$--module. Then $M$ is invertible if and only if its restriction to any finite subgroup is invertible.
\end{proposition}

\begin{proof}

Let $f\colon M\to N$ be a homomorphism of $kG$--modules. Consider the distinguished triangle $$M\xrightarrow[]{f} N\to C(f)\to\Omega^{-1}M.$$ Then $f$ is an isomorphism in $\underline{\StMod}(kG)$ if and only if its cone $C(f)$ is trivial, that is, if it has finite projective dimension. Since $G$ is of type $\Phi$, finite projective dimension is detected by the family of finite subgroups, and we have that $f$ is an isomorphism in $\underline{\StMod}(kG)$  if and only if $f\restr_H$ induces an isomorphism in $\underline{\StMod}(kH)$ for any finite subgroup $H$ of $G$.

Consider the evaluation map $\textrm{ev}\colon M\otimes M^*\to k$. If $M\restr_H$ is invertible, it is well known that $\textrm{ev}\restr_H$ induces an isomorphism in $\underline{\StMod}(kH)$, and then it follows that $M$ is invertible. 
\end{proof}

\subsection{Descent} 

Let $(\C,\otimes,\mathbf{1})$ be a stable homotopy theory, that is, a presentable, symmetric monoidal stable $\infty$-category, where the tensor product is cocontinuous in each variable. A full subcategory of $\C$ is called \textit{thick} if it is closed under finite colimits and limits, and retracts. Further, it is called a \textit{thick $\otimes$-ideal} if in addition  it is a $\otimes$-ideal.

A commutative algebra object $A$ of $\C$ \textit{admits descent} if $\C$ is the smallest thick $\otimes$-ideal containing $A$. The following result can be found in \cite{Mat16} as Proposition 3.22.

\begin{proposition}
    Let $A$ be a commutative algebra object in $\C$ that admits descent. Then the adjunction given by tensoring with $A$ and forgetting is comonadic. In particular, the natural functor to the totalization $$\C\to \mathrm{Tot}\left( \Mod(A)\doublerightarrow{}{}\Mod(A\otimes A)\triplerightarrow{}{}...\right)$$ is an equivalence.

\end{proposition}


\section{The Stable Module $\infty$-Category}

In this section we will closely follow Benson's ideas \cite[Section 10]{Ben97} to  show that the category of $kG$--modules admits a combinatorial symmetric monoidal model structure, for any group $G$ of type $\Phi$. While Benson's treatment deals with a more general class of groups, some restrictions on the modules are needed. On the other hand, working with groups of type $\Phi$ will allow us to consider all $kG$--modules. We also emphasize that it is not known whether the class of groups of type $\Phi$ is completely contained in the class of groups that appear in Benson's work, this is related to a conjecture by O. Talelli \cite[Conjecture A]{Tal}. However, most of the definitions and proofs are more or less the same. The main difference is that for groups of type $\Phi$, the model structure turns out to be combinatorial and symmetric monoidal. We are not aware if the same is known in the setting of \cite{Ben97}. 

This result extends  the well known case for finite groups where the model structure is given as follows: the fibrations are the surjective homomorphisms, the cofibrations are the injective homomorphisms, and the weak equivalences are stable isomorphisms, that is, isomorphisms in $\underline{\Mod}(kG)$. We will use this model structure on $\Mod(kG)$ to define the stable module $\infty$-category $\StMod(kG)$.

\begin{definition}
A $kG$--module $M$ is said to be \textrm{cofibrant} if $B\otimes_k M$ is a projective module, where $B$ is the set of functions from $G$ to $k$ which take only finitely many different values in $k$, with the $kG$--structure given by pointwise scalar multiplication and addition, and $G$--action given by $g(\varphi)(h)=\varphi(g^{-1}h)$.  
\end{definition}

\begin{lemma}\label{finite pd is projective}
    Let $M$ be a cofibrant $kG$--module. If $M$ has finite projective dimension, then $M$ is projective.
\end{lemma}
 \begin{proof}
     Let $\mu\colon B\otimes B\to B$ denote the pointwise multiplication map, and let $i\colon k\to B$ denote the inclusion of the constant functions. Let $\Omega^{-1} M $ denote $\Coker(i)\otimes M$, that is, the cokernel of the injective map $$M\xrightarrow{\cong} k\otimes M \xrightarrow{i\otimes 1} B\otimes M$$ Since $M$ is cofibrant, by Schanuel's lemma we have $$M\oplus P\cong \Omega \Omega^{-1} M \oplus Q$$ for some projective modules $P,Q$. Note that the exact sequence obtained by tensoring with $B$  $$0\to B\otimes M\to B\otimes B\otimes M \to B\otimes \Omega^{-1}M \to 0$$ has a splitting given by $\mu\otimes 1_M$, hence $\Omega^{-1}M$ is cofibrant. We define inductively  $\Omega^{-n}M$ by $ \Omega^{-1}\Omega^{-n+1}M$, for $n\geq2$. In particular, $\Omega^{-n}M$ is cofibrant for all $n\geq1$.  

     Suppose that $M$ has finite projective dimension at most $r>0$. Then $\Omega^r M$ is projective. On the other hand, note that $\Omega^{-r}M$ has projective dimension at most $r$. Consider a projective resolution of $\Omega^{-r}M$  $$0\to \Omega^{r}(\Omega^{-r}M) \to P_{r-1}\to...\to P_0\to \Omega^{-r}M$$ By definition of $\Omega^{-r}M$, we can construct an exact sequence $$0\to M \to Q_{r-1}\to...\to Q_0\to \Omega^{-r}M$$ with $Q_i$ projective, for $i\in\{0,...,r-1\}$. By the extended version of Schanuel's Lemma we have that $\Omega^{r}\Omega^{-r}M\simeq M$ up to projectives, thus $M$ is projective as well. 
 \end{proof}

\begin{lemma}\label{8.5 Benson}
Let $G$ be a group of type $\Phi$. Let $M$, $N$ be cofibrant  $kG$--modules, then the natural map $$\underline{\Hom}_{kG}(M,N)\to \widehat{\Ext}^0_{kG}(M,N)$$ is an isomorphism. 
\end{lemma}

\begin{proof}
Note that for any cofibrant module $M$, we have isomorphisms $$M\simeq \Omega\Omega^{-1}M\simeq \Omega^{-1}\Omega M$$ up to projectives
   (with the same notation as in the proof of Lemma \ref{finite pd is projective}). Hence the result follows.  
\end{proof}

The following definition corresponds to \cite[Definition 10.1]{Ben98II}. 

\begin{definition}\label{model structure}
Let $G$ be a group of type $\Phi$, and $f\colon M\to N$ a homomorphism of $kG$--modules. We say that $f$ is: 
\begin{itemize}
    \item[(i)] a fibration if it is surjective,
    \item[(ii)] a cofibration if it is injective with cofibrant cokernel,  
    \item[(iii)] a weak equivalence if it is a stable isomorphism, that is, if it is an isomorphism in $\underline{\StMod}(kG)$. 
\end{itemize}
\end{definition}

In particular, if $G$ is finite, then any $kG$--module is cofibrant. Moreover, by Lemma \ref{8.5 Benson} we have that $\widehat{\Ext}^0_{kG}(M,N)$ is just $\underline{\Hom}_{kG}(M,N)$. Hence Definition \ref{model structure} coincides with the model structure on $\Mod(kG)$ given in \cite[Section 2.2]{Hov}.

\begin{lemma}\label{characterization of trivial fibrations}
The following holds for a map $f$ of $kG$--modules.    
\begin{itemize}
    \item[(i)] $f$ is a trivial cofibration if and only if $f$ is injective with projective cokernel.
    \item[(ii)] $f$ is a trivial fibration if and only if $f$ is surjective and the kernel has finite projective dimension.
\end{itemize}
\end{lemma}

\begin{proof}
$(i)$ Assume that $f$ is injective. Consider the exact sequence $$0\to M\xrightarrow{f} N \to L \to 0.$$ It defines a distinguished triangle in the stable module category.  If $f$ is a trivial cofibration, then the cofibrant module $L$ is trivial in the stable module category, hence it has finite projective dimension. Thus $L$ is projective by Lemma \ref{finite pd is projective}. On the other hand, if $f$ has projective cokernel, then it is clear that $f$ induces a stable isomorphism. Moreover, any projective is cofibrant, hence $f$ is a trivial cofibration.  $(ii)$ It follows in a similar fashion.   
\end{proof}

\begin{definition}\label{generation cofibrations}
Let $G$ be a group of type $\Phi$. Let $\mathcal{J}$ be the set consisting of the inclusion $0\to kG$, and let $\mathcal{I}$ be the set containing all the induced maps $I\ind_F^G \to kG$, where $I\to kF$ is an inclusion of a left ideal $I$ of $kF$, and $F$ is a finite subgroup of $G$. 
\end{definition}

\begin{proposition}
The class $\mathcal{J}\textrm{-inj}$ is given by the class of fibrations. 
\end{proposition}

\begin{proof}
Let $f\colon M\to N$ be a map in $\Mod(kG)$. Suppose that $f$ has the right lifting property with respect to the inclusion $0\to kG$.
For each $n\in N$, there exists an homomorphism $\alpha_n\colon kG\to N$ whose image contains $n$. Thus we have a commutative diagram 

\centerline{
\xymatrix{ 0 \ar[r] \ar[d] & M \ar[d]^f \\ kG \ar[r]^{\alpha_n} & N
}}
\noindent  and let $\gamma\colon kG\to M$ be a filling. It follows that the image of $f\circ \gamma$ contains $n$, hence $f$ is surjective. Conversely, if $f$ is surjective, then the result follows since $kG$ is projective.
\end{proof}

\begin{proposition}
The class $\mathcal{I}\textrm{-inj}$ agrees with the class of trivial fibrations. 
\end{proposition}

\begin{proof}

Let $f\colon M\to N$ be a map in $\Mod(kG)$. Suppose that $f$ is in $\mathcal{I}\textrm{-inj}$. Then $f$ must be a surjection since $0\to kG$ is in $\mathcal{I}$. Since $G$ is of type $\Phi$, it is enough to prove that $\Ker(f)$ is projective on the restriction to any finite subgroup $F$ of $G$. Hence it is enough to prove that $f\restr_F$ is a trivial fibration on $\Mod(kF)$, which is equivalent to proving that $f\restr_F$ has the left lifting property respect to all the inclusions $I\to kF$ where $I$ is a left ideal of $kF$. Suppose that we have a commutative diagram

\centerline{
\xymatrix{ I \ar[r] \ar[d] & M\restr_F \ar[d]^{f\restr_F} \\ kF \ar[r] & N\restr_F
}}

\noindent By the restriction-induction adjunction we have a commutative diagram

\centerline{
\xymatrix{ I\ind^G \ar[r] \ar[d] & M \ar[d]^f \\ kG \ar[r] & N
}}

\noindent  Since $I\ind^{G}\to kG$ is in $\mathcal{I}$, there is a filling for the latter diagram. Hence by the adjunction there is a filling for the former diagram. The converse follows in a similar fashion. 
\end{proof}

\begin{proposition}
The class of cofibrations agrees with the class $\mathcal{I}\textrm{-cof}$. Moreover, a map is in $\mathcal{J}\textrm{-cof}$ if and only if it is a trivial cofibration. 
\end{proposition}
 
 \begin{proof}
 It follows in the same fashion as Proposition 10.8 in \cite{Ben97}.
 \end{proof}

\begin{theorem}\label{combinatorial structure}
Let $G$ be a group of type $\Phi$. Then $\Mod(kG)$ is a combinatorial model category with respect to the collections of cofibrations, fibrations and weak equivalences described in Definition \ref{model structure}. Moreover, the functor $\Mod(kG)\to \underline{\StMod}(kG)$ induces an equivalence $\Ho\Mod(kG)\to \underline{\StMod}(kG)$. 
\end{theorem}

\begin{proof}
The category of $kG$--modules is locally presentable since it can be identified with the functor category  $\mathrm{Fun}(G,\Mod(k))$, and $\Mod(k)$ is locally presentable, where $G$ is viewed as a category with just one object.  Note that every $kG$--module is small (see  \cite[Example 2.1.6]{Hov}). Then the domains of  $\mathcal{I}$ (resp. $\mathcal{J}$) are small relative to $\mathcal{I}\textrm{-cell}$ (resp. $\mathcal{J}\textrm{-cell}$). Let $\mathcal{W}$ denote the class of weak equivalences. We had proved that $\mathcal{J}\textrm{-cell}\subseteq \mathcal{W}\cap\mathcal{J}\textrm{-cof}$, and $\mathcal{I}\textrm{-inj}\subseteq \mathcal{W}\cap\mathcal{J}\textrm{-inf}$, and $\mathcal{W}\cap\mathcal{I}\textrm{-cof}\subseteq \mathcal{J}\textrm{-cof}$. Hence the result follows from  \cite[Theorem 2.1.19]{Hov}.
\end{proof}
 
For groups of type $\Phi$, the class of cofibrant modules coincides with the class of Gorenstein projective modules (see \cite[Conjecture A]{DT10} and \cite{BDT09}), and by the definition of fibration, any module is fibrant. Thus the full subcategory of $\underline{\StMod}(kG)$ on bifibrant\footnote{An object that is both fibrant and cofibrant.} modules agrees with $\underline{\mathrm{GP}}(kG)$. Then by general theory of model categories we have the following result.

\begin{corollary}
The map $\underline{\Hom}(M,N)\to \Hom_{\underline{\StMod}(kG)}(M,N)$ is surjective if $M$ is Gorenstein projective. Moreover, the inclusion of $\underline{\textrm{GP}}(kG)$ into $\underline{\StMod}(kG)$ is an equivalence.  
\end{corollary}

\begin{proposition}\label{monoidal model category on Mod}
Consider $\Mod(kG)$ endowed with the symmetric monoidal structure given by the tensor product over $k$. Then this monoidal structure and the model structure from Definition \ref{model structure} make $\Mod(kG)$ a symmetric monoidal model category. 
\end{proposition}

\begin{proof}
Recall that for any $kG$--module $M$, the functor $-\otimes M$ is exact. In particular, let $\alpha\colon Qk \xrightarrow{\sim} k$ be the cofibrant replacement of the unit $k$, then we have that $Qk\otimes M \xrightarrow{\sim} k\otimes M$ is a weak equivalence.  

On the other hand, since cofibrations are injective maps with Gorenstein projective cokernel, and tensoring with a Gorenstein projective is Gorenstein projective, we deduce that if $f$ and $g$ are cofibrations then the pushout $f\Box g$ is a cofibration. Moreover, if $f$ is one of the generating cofibrations and $g$ is one of the generating trivial cofibrations, then $f\Box g$ is a trivial cofibration because it is a trivial cofibration on the restriction to any finite subgroup of $G$. Then the result follows by  \cite[Corollary 4.2.5]{Hov}.
\end{proof}

Recall that any model category has an associated underlying $\infty$-category (see \cite[Def. 1.3.4.15]{Lur17}).

\begin{definition}
Let $G$ be a group of type $\Phi$. We define \textrm{the stable module $\infty$-category} $\StMod(kG)$ as the   underlying $\infty$-category of the model structure on $\Mod(kG)$ from Definition \ref{model structure} (c.f.  \cite[Def. 2.2]{Mat15}). 
\end{definition}


\begin{proposition}
The stable module $\infty$-category $\StMod(kG)$ inherits the structure of a stable homotopy theory in the language of Mathew, that is, it is a presentable, symmetric monoidal stable $\infty$-category, where the tensor product commutes with colimits in each variable.    
\end{proposition}

\begin{proof}
    By Theorem \ref{combinatorial structure} and Proposition \ref{monoidal model category on Mod} we have that $\mathrm{Mod}(kG)$ is a combinatorial symmetric monoidal model category. Hence $\StMod(kG)$ is  presentable and symmetric monoidal by \cite[Proposition 1.3.4.22]{Lur17} and \cite[Example 4.1.7.6]{Lur17}. Moreover, the tensor product on $\StMod(kG)$ commutes with colimits separately in each variable by \cite[Lemma 4.1.8.8]{Lur17}.
\end{proof}

\begin{remark}
 We let $\underline{\Hom}_G(M,N)$ denote the mapping space between objects $M,N$ in $\StMod(kG)$. The homotopy category of $\StMod(kG)$ corresponds to the stable module category $\underline{\StMod}(kG)$ (see Definition \ref{stable module category M-P}), that is, the category whose objects are $kG$--modules and hom-sets are given by $$\pi_0\underline{\Hom}_G(M,N)\cong \widehat{\Ext}^0_{kG}(M.N).$$
\end{remark}

We conclude this section by showing functoriality of the stable module $\infty$-category with respect to morphisms of groups. In fact, we will show that restriction along a morphism of groups is a right and left Quillen functor. As a consequence, we will obtain adjuntions in the $\infty$-categorical setting (e.g., see \cite{Mazel16}). 

\begin{proposition}
 Let $G$ be a group of type $\Phi$ and consider the model structure on the category of modules discussed in this section.  Let $H$ be a subgroup of $G$. Restriction along the inclusion  $i\colon H\to G$ induces a right (resp. left) Quillen functor $\mathrm{Res}^G_H\colon \Mod(kG)\ to \Mod(kH)$ with left (resp. right) adjoint given by the induction functor (resp. the coinduction functor). In particular, given an element $g\in G$, we can restrict along the conjugation map $c_g\colon {^gH}\to H$, so we obtain a right Quillen functor 
$(c_g)^*\colon \Mod(kH)\to\Mod(k \gH)$.
\end{proposition}

\begin{proof}
    It is well known that restriction-induction determines an adjuntion. Hence it is enough to verify that restriction sends (trivial) fibrations to (trivial) fibrations. It is clear that restriction sends epimorphisms to epimorphisms since it is exact. Moreover, a trivial fibration is an epimorphism with kernel having finite projective dimension by Lemma \ref{characterization of trivial fibrations}, but restriction sends modules of finite projective dimension to modules of finite projective dimension. In a similar fashion, the restriction functor is a left Quillen functor. 
\end{proof}

\begin{corollary}
Let $H$ be a subgroup of $G$. The inclusion $i\colon H\to G$ induces a symmetric monoidal functor $$\mathrm{Res}^G_H=i^*\colon \StMod(kG)\to\StMod(kH)$$ with left adjoint $i_!$ known as induction, and right adjoint $i_*$ known as coinduction. In particular, we have that $\mathrm{Res}^G_H$ preserves all (homotopy) limits and colimits. Moreover, given an element $g\in G$, we can restrict along the right conjugation map $c_g\colon {^gH}\to H$, so we obtain a functor 
$$(c_g)^*\colon \StMod(kH)\to\StMod(k \gH).$$     
\end{corollary}


\section{Groups Acting on Trees}\label{section 3}

In this section, we let $G$ be a group of type $\Phi$ acting on a tree $T$. In particular, $G$ corresponds to the fundamental group of a graph of groups $(G(-),\Gamma)$ (see \cite[Definition 3.1]{DD89}). We shall show that the stable module $\infty$-category of $G$ admits a decomposition in terms of the graph $\Gamma$. 

\begin{remark}\label{De grafo a categoria}
    Let $\Gamma$ be a directed graph. We will consider $\Gamma$ as a category, still denoted by $\Gamma$, in the following fashion.  
     \begin{itemize}
        \item The objects are the vertices and edges of the graph $\Gamma$. 
        \item The morphisms are given by the incidence maps and the identities. That is, for an edge $e$ we have morphisms $i(e)\to e$ and $\tau(e)\to e$. 
    \end{itemize}
       The category $\Gamma^\textrm{op}$ associated to a directed graph $\Gamma$ is sometimes referred to as the exit path category of the directed graph.
\end{remark}

For instance, consider an amalgamated product $G=A\ast_C B$ of finite groups. The graph of groups $(G(-),\Gamma)$ can be depicted by $A\xleftarrow[]{i}C \xrightarrow[]{j} B$. The associated graph $\Gamma$ (as a category) corresponds to the barycentric subdivision of a segment. We have a corresponding diagram of shape $\Gamma^\textrm{op}$ in $\textrm{Cat}^\otimes_\infty$ depicted as follows. 

\centerline{\xymatrix{
 & \StMod(kA) \ar[d]^{i^*} \\
 \StMod(kB) \ar[r]^{j^*} & \StMod(kC)}}

\noindent In this case, we will show that $\StMod(kG)$ is the pullback in $\textrm{Cat}^\otimes_\infty$ of the diagram above.
Let $\D$ be the homotopy pullback of the diagram  $\Gamma^\textrm{op}\to\textrm{Cat}^\otimes_\infty$, 
and let $\Psi\colon\StMod(kG)\to \D$ denote the comparison functor. In order to show that $\Psi$ is essentially surjective, consider a $kA$--module $M$ and a $kB$--module $N$ such that $M\simeq_\varphi N$ in $\underline{\StMod}(kC)$. Following \cite[Section 5]{Sym}, we can add projectives to $M$ and $N$ so that $\varphi$ can be a genuine isomorphism of $kC$--modules. Let $C(\varphi)= M$ as a $k$-vector space. Fix $m\in C(\varphi)$. For $a\in A$ define $a\cdot m= am$ and for $b\in B$ define $b\cdot m= \varphi^{-1}(b\varphi(m))$. This action makes $C(\varphi)$ a $kG$--module. Moreover, note that $C(\varphi)|_A= M$ and $C(\varphi)|_B\simeq N$.  Therefore, $\Psi$ is essentially surjective. On the other hand, consider $M,N\in \StMod(kG)$. It is enough to show that $$\pi_\ast\underline{\Hom}_G(M,N)\to \pi_\ast\Hom_\D(\Psi(M),\Psi(N))$$ is an isomorphism. Since $\D$ is the pullback of the diagram $\Gamma^\textrm{op}\to \textrm{Cat}^\otimes_\infty$, we have that the space $\Hom_\D(\Psi(M),\Psi(N))$ corresponds to the homotopy pullback of a diagram of shape $\Gamma^\textrm{op}$. In particular, we have a long exact sequence 
\begin{align*}
...\to \pi_{n}\Hom_\D(\Psi(M),\Psi(N)) \to \pi_n \underline{\Hom}_A(M,N)\times \pi_n  \underline{\Hom}_B(M,N) \to \\ \to \pi_n \underline{\Hom}_C(M,N)\to...
\end{align*}

Moreover, recall that $\pi_\ast\underline{\Hom}_G(M,N)$ is given by complete cohomology. By \cite[Section VII.9]{Bro} we have a similar long exact sequence to compute the homotopy groups $\pi_\ast\underline{\Hom}_G(M,N)$. Then we can compare both long exact sequences and by the 5-lemma, we get the desired isomorphism. Hence $\Psi$ is fully faithful and $\StMod(kG)\simeq \D$. \\

We will extend this idea to exhibit a decomposition of the stable module $\infty$-category for any group of type $\Phi$ acting on a tree. Let $D\colon I^\textrm{op} \to \textrm{Top}$ be a diagram. Recall that there is a spectral sequence for computing the homotopy groups of the homotopy limit of $D$ given by
 $$E^{p,q}_2=H^p(I;\pi_q D)\Rightarrow \pi_{q-p}(\varprojlim D)$$ where $\pi_q D$ denotes the diagram of shape $I$ given by $i\mapsto \pi_q(D_i)$. The differentials have the form $d^r\colon E_r^{p,q}\to E_r^{p+r,q+r-1}$. 
Since this spectral  sequence has certain convergence issues, we have to place some restrictions in order to avoid these inconveniences. For instance, we can impose that  each $D_i$ is path connected with abelian fundamental group for each $i$ to ensure convergence (see \cite{BK72}, \cite{Dug}). On the other hand, if our diagram $D$ takes values in spectra $\mathrm{Sp}$, we can use the analogous spectral sequence given in \cite[Section 1.2.2]{Lur17} which does not have these convergence issues.

\begin{remark}\label{reduced complex}
    Let $\Gamma$ be a directed graph. Let $V\Gamma$ denote the set of vertices and $E\Gamma$ denote the set of edges of $\Gamma$. Let $D\colon \Gamma\to \textrm{Ab}$ be a diagram ($\Gamma$  viewed as a category, see Remark \ref{De grafo a categoria}). Let $\overline{C}^*(\Gamma;D)$ be the following 2-term complex of abelian groups  
$$ \prod_{v\in V\Gamma} D(v) \xrightarrow{d}  \prod_{e\in E\Gamma} D(e) $$
where the differential $d$ is given as follows. For $U$ in the domain, 
$$d(U)(e)= D(v\to e)U(v)-D(w\to e)(U(w))$$ where $e$  is an edge of the directed graph with initial vertex $v$ and terminal vertex $w$, and $v\to e$ and $w\to e$ are given by the incidence functions of the graph. Note that $\overline{C}^*(\Gamma;D)$ is quasi-isomorphic to the cochain complex $C^*(\Gamma;D)$ of $\Gamma$ with coefficients in $D$. In particular, we can use $\overline{C}^*(\Gamma;D)$ to compute $H^*(\Gamma;D)$.

\end{remark}

\begin{theorem}\label{graphs of groups}
    Let $G$ be a group of type $\Phi$ acting on a tree $T$. Consider the associated graph of groups $\Gamma \to \textrm{Gps}$. Then we have an equivalence of symmetric monoidal $\infty$-categories 
$$\StMod(kG)\xrightarrow[]{\simeq} \varprojlim_{\sigma\in \Gamma^\textrm{op}} \StMod(kG_\sigma). $$ 
    
\end{theorem}

\begin{proof}

Consider the canonical functor $$\Psi\colon\StMod(kG)\to \displaystyle \varprojlim_{\sigma\in \Gamma^\textrm{op}} \StMod(kG_\sigma).$$  Let $\C$ denote the right-hand side $\infty$-category. Fix $M,N$ in $\StMod(kG)$. Let $D$ be the diagram of shape $\Gamma^\textrm{op}$ that maps $\sigma$ to the mapping space $\underline{\Hom}_{G_\sigma}(M_\sigma,N_\sigma)$.  The Bousfield-Kan spectral sequence for homotopy limits $$E^{p,q}_2=H^p(\Gamma;\pi_{-q} D)$$ converges to the homotopy groups $$\pi_{-(p+q)}\underline{\Hom}_\C(\Psi(M),\Psi(N))$$ 
Recall that the homotopy groups $\pi_i\underline{\Hom}_H(X,Y)$ are given by the complete cohomology $\widehat{\Ext}^{-i}_H(X,Y)$. On the other hand, we have a spectral sequence (see \cite[Chapter VII]{Bro}) to compute the homotopy groups of the mapping space $\underline{\Hom}_G(M,N)$ given by 
$$\widetilde{E}_1^{p,q}= \prod_{\sigma \in T_q}\widehat{\Ext}^p_{kG_\sigma}(M,N)\Rightarrow \widehat{\Ext}^{p+q}_{kG}(M,N)$$ where $T_q$ is a set of representatives of the $G$--orbits of $q$-simplices of $T$. In particular $T_0$ is in bijection with  $V\Gamma$ and $T_1$ is in bijection with $E\Gamma$. We claim that the induced map $\pi_n(\Psi)\colon \pi_n\underline{\Hom}_G(M,N)\to \pi_n\underline{\Hom}_\C(\Psi(M),\Psi(N))$ is an isomorphism for all $n\in\mathbb{Z}$. First, note that the two spectral sequences $\widetilde{E}^{p,q}$ and $E^{p,q}$ look the same at the page 2. We will show that $\pi_\ast(\Psi)$ is compatible with the the obvious comparison map between the spectral sequences. By examining both spectral sequences, we obtain short filtrations
\begin{align*}
    0=\widetilde{F}^2\subseteq\widetilde{F}^1=\widetilde{E}^{1,n}_\infty\subseteq \widetilde{F}^0=  \pi_n\underline{\Hom}_G(M,N) \\
    0=F^2\subseteq F^1=E^{1,n}_\infty \subseteq F^0=\pi_n\underline{\Hom}_\C(\Psi(M),\Psi(N)). 
\end{align*}

\noindent Since the map $\Psi$ is induced by the restriction functors, it follows that $\pi_n(\Psi)$ is compatible with $\widetilde{E}^{0,n}_\infty\to E^{0,n}_\infty$. It remains  to verify that  $\pi_\ast\Psi$ is compatible with the filtrations $\widetilde{F}$ and $F$. Once again using that $\Psi$ is induced by the restriction functors and comparing $\widetilde{E}^{p,q}_1$ with the $p$-term of the complex $\overline{C}^\ast(\Gamma, \pi_{-q}\circ D)$ which computes $H^p(\Gamma,\pi_{-q}\circ D)$ (see Remark \ref{reduced complex}), we deduce that the map $\widetilde{E}^{1,n+1}_\infty\to E^{1,n+1}_\infty$ is compatible with $\pi_n(\Psi)$. Therefore the claim follows and we obtain that $\Psi$ is fully faithful.  By \cite[Lemma 7.1]{MS}, we have that $\Psi$ is essentially surjective, and hence an equivalence. 
\end{proof}

\begin{remark}\label{construction of the left adjoint}
In fact, following \cite{HY} we can describe a left adjoint $\Upsilon$ of $\Psi$ (hence an inverse)  as the composition $$\C \xrightarrow[]{\Upsilon^{\Gamma^\textrm{op}}} \StMod(kG)^{\Gamma^{\textrm{op}}} \xrightarrow[]{\lim } \StMod(kG)$$ where $\lim$ is the left adjoint of the constant diagram functor, and $\Upsilon^{\Gamma^\textrm{op}}$ is the left adjoint of the induced functor on the lax limit. In other words, the functor $\Upsilon$ can be described informally by the formula $$\Upsilon((M_\sigma)_{\sigma\in \Gamma})= \lim_{\sigma\in \Gamma^\textrm{op}} (M\uparrow_{G_\sigma}^G ). $$    
\end{remark}


\section{Groups Admitting a Finite-Dimensional Model for $\underline{E}G$}

In this section, we will assume that $G$ admits a finite-dimensional model for the classifying space for proper actions $\underline{E}G$. It  has been conjectured that a group $G$ is of type $\Phi$ over $\mathbb{Z}$ if and only if it admits a finite-dimensional model for $\underline{E}G$ (see \cite[Conjecture A]{Tal}). We will follow  Balmer's ideas in \cite{Bal16} and adapt Mathew's proof of the decomposition of the stable module $\infty$-category for finite groups over the orbit category (see \cite[Section 9]{Mat16}) to exhibit an analogous decomposition for $\StMod(kG)$ in terms of the stable module $\infty$-category of its finite subgroups.

\subsection{A monadic adjunction}

\begin{proposition}
   Let $H$ be a subgroup of $G$. Then the adjunction given by $$\mathrm{Res}\colon \Mod(kG)^{\textrm{op}}\leftrightarrows \Mod(kH)^{\textrm{op}} \colon \mathrm{Ind}$$ is monadic. In particular, we have that $\Mod(kH)^\textrm{op}$ is equivalent to the category $\Mod_{\Mod(kG)^\textrm{op}}(A_H')$ of $A_H'$--modules in $\Mod(kG)^{\textrm{op}}$, where $A_H'=\mathrm{Ind}\circ \mathrm{Res}$.
\end{proposition}

\begin{proof}
    By \cite[Lemma 2.10]{Bal15} it is enough to exhibit a natural section of the counit $\epsilon\colon 1_{\Mod(kH)} \to \mathrm{Res}\circ\mathrm{Ind}$ (since it is the opposite adjuntion, we are writing the unit of the adjunction $\mathrm{Ind}-\mathrm{Res}$). Recall that the $M$-component of the counit is given by $m\mapsto 1\otimes m$ for $m\in M$. For a $kH$--module $M$, we define  $$\psi_M\colon\mathrm{Res}(\mathrm{Ind}(M))\to M$$ as the map 
   $$ g\otimes m \mapsto \left\{
        \begin{array}{ll}
            gm & \textrm{if } g\in H \\
         
            0 & \textrm{if } g\not\in H.
        \end{array}
    \right.
    $$
Note that this defines a natural transformation $\psi\colon\mathrm{Res}\circ\mathrm{Ind}\to 1_{\Mod(kH)}$ such that $\epsilon\psi=1$, hence the result follows. 
\end{proof}

\begin{definition}\label{Induced algebra object}
Let $H$ be a subgroup of $G$. Let $A_H$ denote the $kG$--module $k(G/H)$. Define a comultiplication $\mu\colon A_H\to A_H\otimes A_H$ by $\gamma\mapsto \gamma\otimes \gamma$, and a counit $\epsilon\colon A_H\to k$ by the augmentation map.       
\end{definition}

\begin{proposition}\label{Monad induction-restriction 1-cat}
    Let $H$ be a subgroup of $G$. We have that $(A_H,\mu,\epsilon)$ defines a separable algebra object on the symmetric monoidal category $\Mod(kG)^\textrm{op}$. Moreover, the monad $A_H'$ induced by the adjunction   $$\mathrm{Res}\colon \Mod(kG)^{\textrm{op}}\leftrightarrows \Mod(kH)^{\textrm{op}} \colon \mathrm{Ind}$$ is isomorphic to the monad induced by $A_H\otimes-$.
\end{proposition}

\begin{proof}
    It is straightforward to verify that $(A_H,\mu,\epsilon)$ defines a coalgebra object in $\Mod(kG)$ and therefore an algebra object in $\Mod(kG)^\textrm{op}$. The separability of $A_H$ will follow from the equivalence, as monads, with $A'_H$. Recall that we have a natural isomorphism of functors $$\theta\colon \mathrm{Ind}\circ \mathrm{Res} \to A_{H}\otimes-$$ where the $M$-component is given by $\theta_{M}(g\otimes m)= gH\otimes gm$ and the inverse is given by $\theta_M^{-1}(\gamma\otimes m)= g\otimes g^{-1}m$ for any choice of $g\in\gamma$. Note that we have a compatibility of the units, that is, $\theta_M\circ {\epsilon'}_M=\epsilon\otimes1_M$ for any $kG$--module $M$. The multiplication $\mu'$ of the monad is given by $\mu'(g\otimes m)= g\otimes 1\otimes m$. Then the following diagram is commutative.

\centerline{\xymatrix{
kG\otimes_{kH}(kG\otimes_{kH}M)  \ar[r]^{\theta^2_M} & k(G/H)\otimes k(G/H)\otimes M \\ 
kG\otimes_{kH} M \ar[u]_{\mu} \ar[r]_{\theta_M} & k(G/H)\otimes M \ar[u]^{\mu'} }}

The map $\theta^2_M$ is given as follows.  For any $g,g'\in G$, we have that $\theta_M^2(g\otimes g'\otimes m)=gH\otimes gg' H\otimes gg'm$. Then the result follows. 
\end{proof}

Let $H$ be a subgroup of $G$, and consider the image of $A_H$ in the stable module category, still denoted by $A_H$. We will show that the analogous result in the $\infty$-categorical setting  of the previous proposition also holds.

\begin{proposition}\label{Monad induction-restriction}
Let $H$ be a subgroup of $G$. Let $A_H\in \CAlg(\StMod(kG)) $. Then there is an equivalence $$\Mod_{\StMod(kG)^{\textrm{op}}}(A_H)\simeq \StMod(kH)^{\textrm{op}}$$ and we can identify the adjunction $\StMod(kG)^{\textrm{op}}\leftrightarrows \Mod_{\StMod(kG)^{\textrm{op}}}(A_H)$ with the adjunction $\mathrm{Res}\colon\StMod(kG)^{\textrm{op}}\leftrightarrows \StMod(kH)^{\textrm{op}}\colon \mathrm{Ind}$.    
\end{proposition}

\begin{proof}
    By \cite[Theorem 2.5]{Ram23}, we have that the forgetful functor $$\mathrm{Ho}\Mod_{\StMod(kG)^{\textrm{op}}}(A_H)\to \Mod_{\underline{\StMod}(kG)^{\textrm{op}}}(A_H)$$ is an equivalence. Moreover, the functor $\mathrm{Ho}\Mod_{\StMod(kG)^{\textrm{op}}}(A_H)\to \underline{\StMod}(kH)$ factors through the previous equivalence. Hence by Proposition \ref{Monad induction-restriction 1-cat} we deduce the result. 
\end{proof}


\subsection{Decompositions of the stable module $\infty$-category}\label{subsection 4.2}

Recall that an action of a group $G$ on a space  $X$ is called cocompact if the orbit space is compact. Moreover, a group $G$ has a cocompact model $X$ for $\underline{E}G$, if the action of $G$ on $X$ is cocompact.

\begin{proposition}\label{A cal admits descent}
Let $G$ be a group with a finite-dimensional cocompact model $X$ for $\underline{E}G$. Let $\mathscr{F}$ be the family of finite subgroups of $G$.  Then the commutative algebra object $$A=\prod_{H\in \mathscr{S}} A_H\in \CAlg(\StMod(kG)^\textrm{op})$$ admits descent, where $\mathscr{S}$ is a set of representatives of the $G$-orbits of $\mathscr{F}$.
\end{proposition}

\begin{proof}
 Let $C_*(X)$ denote the chain complex of $X$ with coefficients in $k$.  By the hypothesis on $X$, the set $\mathscr{S}$ is finite. Since the forgetful functor $$\CAlg(\StMod(kG)^\textrm{op})\to\StMod(kG)^\textrm{op}$$ commutes with limits (see \cite[Proposition 3.2.2.1]{Lur17}), it follows that $A$ is just a finite product in $\StMod(kG)^\textrm{op}$, and hence a finite coproduct. As a consequence, we have that $C_n(X)$ is a retract of a finite number of copies of $A$. Therefore $C_n(X)$ is  contained in the smallest thick $\otimes$-ideal  $\langle A\rangle $ containing $A$, for any $n\in \mathbb{Z}$. Since $X$ is contractible, we have that the augmented chain complex $\widetilde{C}_*(X)$ is exact, and therefore $k$ is in $\langle A\rangle $. It follows that $\langle A\rangle=\StMod(kG)^{\textrm{op}}$.
\end{proof}


\begin{remark}\label{generalizacion del objeto algebra que admite descenso}
Consider the same notation as in the previous proposition. Let $\mathscr{S}'$ denote the set of Sylow $p$-subgroups of the elements of $\mathscr{S}$. Then the commutative algebra object $$B=\prod_{H\in \mathscr{S}'} A_H\in \CAlg(\StMod(kG)^\textrm{op})$$ has $A$ as a retract, thus $B$ admits descent as well. This follows since $k$ is a retract of $k\ind_S^F$ as $kF$--modules, where $F$ is a finite group and $S$ is a Sylow $p$-subgroup of $F$.  In the same fashion, we can construct commutative algebra objects in   $\CAlg(\StMod(kG)^\textrm{op})$ which admit descent as long as the set of subgroups indexing our commutative algebra object contains a copy of representatives of the Sylow $p$-subgroups of the elements in $\mathscr{S}$. 
\end{remark}


Recall that the \textit{orbit category} $\mathcal{O}(G)$ is the category with objects the $G$-sets of the form $G/H$ where $H$ is a subgroup of $G$, and the morphisms are given by $G$-maps. Given a collection $\mathcal{A}$ of subgroups of $G$, that is, a set of subgroups of $G$ closed under conjugation, we let $\mathcal{O}_{\mathcal{A}}(G)\subseteq \mathcal{O}(G)$ denote the full subcategory spanned by the objects $G/H$ with $H\in  \mathcal{A}$. 

\begin{proposition}\label{decomposition of C}
   Let $G$ be a group with a finite-dimensional cocompact model $X$ for $\underline{E}G$. Let $\mathscr{F}$ be a family of finite subgroups of $G$ which contains the family of finite $p$-subgroups of $G$. Then there is an equivalence of symmetric monoidal stable $\infty$-categories $$\StMod(kG)\xrightarrow[]{\simeq} \varprojlim_{G/H\in\mathcal{O}_\mathscr{F}(G)^\textrm{op}} \StMod(kH). $$
\end{proposition}

\begin{proof}
Let $A=\prod_{H\in \mathscr{S}} A_H$ be a commutative algebra object in $\StMod(kG)^\textrm{op}$ as in Proposition \ref{A cal admits descent}. Since $A$ admits descent, Proposition 3.22 in \cite{Mat16} gives us a decomposition $$\StMod(kG)^{\textrm{op}}\simeq \textrm{Tot} \left( \Mod_{\StMod(kG)^{\textrm{op}}}(A)\doublerightarrow{}{} \Mod_{\StMod(kG)^{\textrm{op}}}(A^{\otimes2})\triplerightarrow{}{}\ldots \right )$$
Consider the smallest full subcategory $\C$ of $$ \mathcal{P}(\mathcal{O}_\mathscr{F}(G))=\textrm{Fun}(\mathcal{O}_\mathscr{F}(G)^\textrm{op},\mathcal{S})$$    that contains the essential image of the Yoneda embedding $$\mathcal{O}_\mathscr{F}(G)\xrightarrow[]{y} \mathcal{P}(\mathcal{O}_\mathscr{F}(G))$$ and which is stable under finite coproducts (see \cite[Remark 5.3.5.9]{Lur09}). Then we can extend the stable module $\infty$-category functor

  \[
\begin{split}
 f_0\colon \mathcal{O}_\mathscr{F}(G)^\textrm{op} & \to \widehat{\textrm{Cat}}^\otimes_\infty \\
G/H & \mapsto \mathrm{Mod}_{\StMod(kG)^\textrm{op}}(A_H)\simeq \StMod(kH)^\textrm{op}
\end{split}
\] 
to a functor $$f\colon \C^\textrm{op}\to \widehat{\textrm{Cat}}^\otimes_\infty$$ 
\noindent that sends finite coproducts to limits, where the equivalence involved is given in Proposition \ref{Monad induction-restriction}. Moreover, the functor $f$ is the right Kan extension of $f_0= f|_{\mathcal{O}_\mathscr{F}(G)^\textrm{op}}$. Thus we have an equivalence  $$\varprojlim_{\C^\textrm{op}}f\simeq \varprojlim_{\mathcal{O}_\mathscr{F}(G)^\textrm{op}} f_0.$$ In particular, the object $\bigsqcup_{H\in \mathscr{S}}G/H$ in $\C^\textrm{op}$ is mapped under the functor $f$ to $$\prod_{H\in \mathscr{S}} \StMod(kH)\simeq \mathrm{Mod}_{\StMod(kG)^\textrm{op}}(A).$$   

On the other hand, consider the object $Z=\bigsqcup_{H\in \mathscr{S}} G/H\in \C$. Note that any object $Y\in \C$ admits a map $Y\to Z$. By \cite[Proposition 6.28]{MNN} the simplicial object $Z^{\bullet +1}\colon \Delta^\textrm{op}\to \C$ 
is cofinal. Moreover, note that the cosimplicial diagram $f\circ Z^{\bullet+1}$ is in fact the cobar construction. Hence the result follows. 
\end{proof}


It is worth highlighting  that working with the orbit category of an infinite group might be not easy, hence it is convenient to have decompositions of the stable module $\infty$-category in terms of \textit{simpler} categories.  For this, the $\infty$-categorical version of Quillen's Theorem A will play an important role. 

Recall that $G$ denotes a group with a finite-dimensional cocompact model $X$ for $\underline{E}G$. For the rest of this section, we will assume that $X$ is a simplicial complex, and the action of $G$ on $X$ is simplicial. As before, $\mathscr{F}$ denotes the family of finite subgroups of $G$ and  $\textrm{Stab}$ denotes the family of finite subgroups of $G$ of the form $G_\sigma$, for some simplex $\sigma$ of  $X$.

\begin{proposition}\label{cofinalidad de la inclusion de subgrupos de isotropia}
 The full subcategory  $ \mathcal{O}_\textrm{Stab}(G)$ of the orbit category of $G$ induces a cofinal functor $\mathcal{O}_\textrm{Stab}(G) \to \mathcal{O}_{\mathscr{F}}(G)$. 
\end{proposition}

\begin{proof}
   By Quillen's Theorem A \cite[Corollary 4.1.3.3]{Lur09}  it is enough to verify that $\mathcal{O}_\textrm{Stab}(G)_{(G/H) /}$ has weakly contractible nerve, for all $G/H\in \mathcal{O}_\mathscr{F}(G)$. We will argue as in the first part of the proof of \cite[Proposition 6.31]{MNN}. Note that an element  $G/H\xrightarrow[]{[g]} G/G_\sigma$ in $\mathcal{O}_\textrm{Stab}(G)_{(G/H) /}$ is isomorphic 
  to $G/H\xrightarrow[]{[1]} G/G_{g\sigma }$. Hence the inclusion of the full subcategory of  $\mathcal{O}_\textrm{Stab}(G)_{(G/H) /}$ on objects of the form $G/H\xrightarrow[]{[1]} G/G_\sigma$ induces an equivalence, and it can be identified with the opposite of the poset of subgroups of $\textrm{Stab}$ that contain $H$. Let $\mathscr{S}_{H}$ denote this poset. 
  More generally, for a subgroup $K$ of $G$, consider the poset $\mathscr{S}_K=\{K'\in \textrm{Stab}\mid K\subseteq K'\} $.  We will show that $\mathscr{S}_H$ is contractible by induction on the length $l(H)$ of $H$, that is, the maximum length of an ascending chain $$H\subsetneq H_1\subsetneq\ldots \subsetneq H_n$$ with $H_1,\ldots,H_n$ finite subgroups of $G$. Note that $l(H)$ is finite since there are only a finite number of finite subgroups of $G$ up to conjugacy by our assumption on $G$. If $l(H)=0$, then it follows that $H$ must be in $\textrm{Stab}$ since any finite subgroup is contained in the isotropy group of some simplex. Hence, $\mathscr{S}_H$ has a minimum  element. Now, suppose that $l(H)>0$. If $H$ is in $\textrm{Stab}$, we are done. If not, consider the collection $\{X^K\mid K\in \mathscr{S}_H\}$ which is a cover of $X^H$ since for any point $x\in X^H$, we have that $x\in X^{G_x}$. Similarly, consider the cover $\{|\mathscr{S}_K|\mid K\in \mathscr{S}_H\}$ of $|\mathscr{S}_H|$. Moreover, note that both $X^K$ and $|\mathscr{S}_K|$ are contractible, for any $K\in \mathscr{S}_H$. If we show that for any finite collection of subgroups $K_1,\ldots,K_n$ in $\mathscr{S}_H$, the intersections $Y= X^{K_1}\cap\ldots\cap X^{K_n}$ and $Z=|\mathscr{S}_{K_1}|\cap\ldots\cap |\mathscr{S}_{K_n}|$ are either empty or contractible, then we can invoke Leray's nerve theorem (for instance, see \cite[Theorem 15.21, Remark 15.22]{Koz08}) to show that $\mathscr{S}_H$ is homotopy equivalent to $X^H$, which completes the proof since $X^H$ is contractible. Let $L$ be the subgroup of $G$ generated by $K_1,\ldots,K_n$. If $L$ is infinite, then both $Y$ and $Z$ are empty. But if $L$ is finite, then $Y$ agrees with $X^L$ and $Z$ with the realization of $\mathscr{S}_L$, in particular, $Y$ is contractible. On the other hand, we have that $H\subsetneq L$, and hence $l(L)<l(H)$. Therefore we can apply the inductive step to $\mathscr{S}_L$ to complete the proof.  
\end{proof}

\begin{corollary}\label{decomposition for Stab}
    There is an equivalence of symmetric monoidal stable $\infty$-categories $$\StMod(kG)\xrightarrow[]{\simeq} \varprojlim_{G/H\in\mathcal{O}_\textrm{Stab}(G)^\textrm{op}} \StMod(kH). $$
\end{corollary}


\section{Some Computations}

\begin{definition}
The \textit{Picard space} $\mathrm{Pic}(\C)$ of a symmetric monoidal $\infty$-category $(\C, \otimes, \mathbf{1)}$ is the $\infty$-groupoid of $\otimes$-invertible objects in $\C$ and equivalences between them.    
\end{definition}

In other words, the Picard space is an enhancement of the Picard group, since the former encodes the latter as the connected components, but also keeps track of all higher isomorphisms. It is clear that $\mathrm{Pic}(-)$  defines a functor from the $\infty$-category of symmetric monoidal $\infty$-categories to the $\infty$-category of spaces.  Moreover, we can describe the higher homotopy groups of $\mathrm{Pic}(\C)$ as follows. 
$$
\pi_i\mathrm{Pic}(\C) = \left\{
        \begin{array}{ll}
           
            \textrm{PicGp}(\C) & \textrm{ if } n=0\\
            (\pi_0 \Hom_\C(\mathbf{1},\mathbf{1}))^\times & \textrm{ if } n=1 \\
            \pi_{i-1} \Hom_\C (\mathbf{1},\mathbf{1}) & \textrm{ if } n\geq 2.
        \end{array}
    \right.
$$ Moreover, the Picard space functor $\mathrm{Pic}(-)$ commutes with homotopy limits (see \cite[Proposition 2.2.3]{MS16}). In particular, we will obtain a limit decomposition of the Picard space of the stable module $\infty$-category for certain groups. If $G$ is a group of type $\Phi$ acting on a tree with associated graph $\Gamma$, then Theorem \ref{graphs of groups} gives us a decomposition of the Picard space 
\begin{equation}
\label{eqn:Pic decomposition graph of groups} 
\mathrm{Pic}(\StMod(kG) )\xrightarrow[]{\simeq} \varprojlim_{\sigma\in \Gamma^\textrm{op}} \mathrm{Pic}( \StMod(kG_\sigma) ).    
\end{equation}
 
 We will use the spectral sequence of Bousfield-Kan for homotopy limits to compute the Picard group of the stable module category. As we already mentioned, this spectral sequence have some convergence issues. Fortunately we can avoid those issues in our setting. For instance, as noticed in \cite[Section 2.2]{VdMW2021}, when $\C$ is stable, $\mathrm{Pic(\C)}$ can be viewed as a connective spectrum since it is an $\infty$-group like $\mathbb{E}_\infty$-space and it is called the \textit{Picard spectrum} of $\C$ (still denoted by $\mathrm{Pic}(\C)$). As a consequence, we obtain that  $\mathrm{Pic}(-)$  defines a functor from the $\infty$-category of symmetric monoidal $\infty$-categories $\mathrm{Cat}^\otimes_\textrm{st}$ to non-connective spectra $\mathrm{Sp}_{\geq0}$ which commutes with limits (see \cite[Section 2.2]{VdMW2021} for further details). Hence, given a symmetric monoidal stable $\infty$-category $\C$ as a limit of a diagram $D\colon I^\textrm{op}\to\mathrm{Cat}^\otimes_{\textrm{st}}$, we obtain a spectral sequence 
\begin{equation}\label{spectral sequence for Pic}
E^{p,q}_2=H^p(I;\pi_q \mathrm{Pic} \,D(-))\Rightarrow \pi_{q-p}(\varprojlim \mathrm{Pic}\, D).     
\end{equation}

 Moreover, for the stable module $\infty$-category of a group $G$ of type $\Phi$, the higher homotopy groups of the Picard space can be described as follows.   
$$
\pi_i\mathrm{Pic}(\StMod(kG)) = \left\{
        \begin{array}{ll}
           
            T(G) & \textrm{ if } n=0\\[5pt]
            \widehat{\textrm{Aut}}_G(k) & \textrm{ if } n=1 \\[5pt]
            \widehat{H}^{1-i}(G,k) & \textrm{ if } n\geq 2
        \end{array}
    \right.$$ where $ \widehat{\textrm{Aut}}_G(k)$ denotes the group of automorphisms of $k$ in the stable module category $\underline{\mathrm{StMod}}(kG)$. In particular, for groups acting on trees we can describe the Picard group as an extension of abelian groups  (c.f. \cite[Theorem 7.4]{MS}). 

\begin{corollary}\label{exact sequence for picard groups}

  Let $G$ be a group of type $\Phi$ which acts on a tree. Consider the associated graph of groups $\Gamma\to \textrm{Gps}$. Then we have a short exact sequence of abelian groups 
$$0\to H^1(\Gamma; \pi_1\circ \mathrm{Pic}\circ\StMod(-))\to T(G)\to H^0(\Gamma; \pi_0\circ \mathrm{Pic}\circ\StMod(-))\to 0.$$ where $\StMod(-)$ denotes the stable module $\infty$-category functor associated to the graph of groups, that is, it maps $\sigma\in\Gamma^\textrm{op}$ to $\StMod(kG_\sigma)$  (see Section \ref{section 3}). 

\end{corollary}

\begin{proof}
By Equation  (\ref{eqn:Pic decomposition graph of groups}), the spectral sequence of Equation  (\ref{spectral sequence for Pic}) has the form $$E^{p,q}_2= H^p(\Gamma; \pi_q\circ \mathrm{Pic}\circ\StMod)\Rightarrow \pi_{q-p} \mathrm{Pic}(\StMod(kG)).$$ Since $\Gamma$ is a graph, and hence no non-trivial composable maps, this spectral sequence  is trivial except possibly for $p=0,1$ (see Remark \ref{reduced complex}). Then the spectral sequence collapses at page two and the result follows.  
\end{proof}

\begin{remark}
Note that, by Remark \ref{reduced complex}, $H^0(\Gamma; \pi_0\circ\mathrm{Pic}\circ\StMod(-))$ corresponds to the kernel of the map 
$$\displaystyle \prod_{v\in V\Gamma} T(G_v)    \xrightarrow[]{\mathrm{Res}-\mathrm{Res}_f}   \displaystyle \prod_{e\in E\Gamma} T(G_e)$$ where $\mathrm{Res}-\mathrm{Res}_f$ is defined as follows. Fix $(M_v)_{v\in V\Gamma}\in \prod_{v\in V\Gamma}T(G_v)$. Then $$\Res -\Res_f (M_v)_{v\in V\Gamma}=(N_e)_{e\in E\Gamma}$$ where $N_e= M_{i(e)}\restr_{G_e}-f_e^\ast(M_{\tau(e)})\restr_{G_e}$, and $f^*_e$ is the functor induced by the morphism $ G_e \to G_{\tau(e)}$ in the graph of groups.  On the other hand, $H^1(\Gamma; \pi_0\circ\mathrm{Pic}\circ\StMod(-))$ corresponds to the cokernel of the map $$ \displaystyle \prod_{v\in VY} \widehat{\textrm{Aut}}_{G_v}(k)   \xrightarrow[]{\mathrm{Res}-\mathrm{Res}_f}  \displaystyle \prod_{e\in EY} \widehat{\textrm{Aut}}_{G_e}(k)$$ defined in a similar fashion. This agrees with \cite[Theorem 7.4]{MS}.
\end{remark}

\subsection{Countable locally finite $p$-groups}\label{Computations locally finite p-groups}

Recall that throughout this document, $k$ denotes a field of prime characteristic $p$. In this section we will use Corollary \ref{exact sequence for picard groups} to compute the Picard group of the stable module category for countable locally finite $p$-groups, at least as an abstract abelian group. Recall that a group $G$ is called \textit{a locally finite $p$–group} if every finitely generated subgroup is a finite $p$–group. The following result is well known (see for example \cite[Lemma 1.A.9]{KW}).

\begin{proposition}
Let $G$ be a locally finite group. Then $G$ is countable if and only if there is an ascending chain of finite subgroups $$G_1\leq G_2\leq G_3\leq...$$ such that $$G=\bigcup_{n\geq1}G_n.$$  In this case, we will say that $G_1\leq G_2\leq G_3\leq...$ is a \textit{tower} for $G$.
\end{proposition}

\begin{proposition}\label{prop toral}
Let $G$ be a countable locally finite group. Consider a tower $G_1\leq G_2\leq ...$ of finite subgroups of $G$. Then the following hold. 
\begin{itemize}
    \item[(i)]  $G$ acts on a tree with isotropy groups in the family $\{G_n\}_{n\geq1}$.
    \item[(ii)] If $p$ divides the order of $G_r$ for some $r\geq1$, then
$$ T(G)\cong  \displaystyle \varprojlim  T(G_n) $$ where the maps in the inverse system are given by the restrictions maps.
\end{itemize}

\end{proposition}

In particular, note that part $(i)$ of this proposition gives us that any countable locally finite group is a group of type $\Phi$. 

\begin{proof}
The first part follows in the same fashion as  \cite[Example 3]{Ike}. For convenience we will give the construction of the tree $T$: define the vertex set $VT$ as the disjoint union of the sets  $G/G_n$ for $n\geq1$. The edges are given by the canonical maps $G/G_n\to G/G_{n+1}$, that is, if $mG_n$ is a vertex, then there is an edge from the corresponding vertex to $mG_{n+1}$.

The graph $T$ is path connected since any vertex will be mapped to the trivial coset eventually. Moreover, it is clear that there are no loops, hence $T$ is a tree. The action of $G$ on the tree $T$ is induced by the action of $G$ on $G/G_n$ by multiplication. For an edge $(mG_n,mG_{n+1})$, the action is given by $$g\cdot(mG_n,mG_{n+1})=((gm)G_n,(gm)G_{n+1})$$ for $g\in G$. Note that the stabilizer of the vertex $mG_n$ is isomorphic to $G_n$.

For the second part, note that the fundamental domain for the action of $G$ on $T$ corresponds to a ray. Then the associated graph of groups  $\Gamma\to \mathrm{Gps}$  can be depicted as follows.

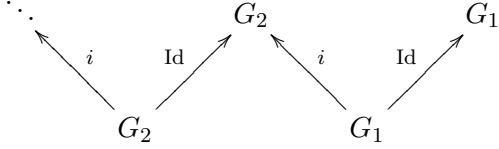
\centerline{\xymatrix{ 
\ddots \ar@{<-}[rd]^{i}  & &  G_2 \ar@{<-}[rd]^{i} \ar@{<-}[ld]_{\textrm{Id}}  & & G_1 \ar@{<-}[ld]_{\textrm{Id}}   \\ & G_2   & &  G_1  &  
}
}

\noindent where $i$ denotes the inclusion $G_n\to G_{n+1}$. Then the diagram $$\StMod(-)\colon \Gamma^{\mathrm{op}}\to \mathrm{Cat}^\otimes_\infty$$ can be depicted by $$...\xrightarrow{\mathrm{Res}}\StMod(kG_3)\xrightarrow{\mathrm{Res}}\StMod(kG_2)\xrightarrow{\mathrm{Res}}\StMod(kG_1)$$ 
By Corollary \ref{exact sequence for picard groups}, we have the following exact sequence of abelian groups. 

$$0\to H^1(\Gamma;\pi_1\StMod)\to T(G) \to H^0(\Gamma;\pi_0\StMod)\to 0$$ 
Note that $\pi_1\StMod(kG_n)\cong\widehat{\textrm{Aut}}_{G_n}(k)\cong k^\times$ for all $n\geq r$, thus $\pi_1\StMod$ is eventually constant. Hence we have that $H^1(\Gamma, \pi_1 \StMod))\cong H^1(|\Gamma|,k^\times)=0$. On the other hand, recall that $$H^0(\Gamma;\pi_0\StMod)=\varprojlim \pi_0(\StMod(kG_n))= \varprojlim T(G_n)$$ and the result follows. 
\end{proof}

We will first consider the case when the group is artinian, that is, if its subgroups satisfy the descending
chain condition. 

\begin{definition}
A group $P$ is called \textit{a discrete $p$-toral group} if it fits in an extension  $$1\to K \to P \to S\to 1$$ where $K$ is isomorphic to a finite product of copies of $\mathbb{Z}/p^\infty$ and $S$ is a finite $p$-group. 
\end{definition}

Note that  $\mathbb{Z}/p^\infty$ is an artinian locally finite $p$-group. Since these properties are preserved by finite products and finite extensions, we deduce that any discrete $p$-toral group is an artinian locally finite $p$-group. The converse also holds (see \cite[Proposition 1.2]{BLO}). Hence a group is a locally finite $p$-group if and only if it is a discrete $p$-toral group. As a  consequence of this characterization, it follows that the class of discrete $p$-toral groups is closed under subgroups, quotients and extensions by  discrete $p$-toral groups. Moreover, if  $P$ is a discrete $p$-toral group, then it contains finitely many conjugacy classes of elementary abelian $p$-subgroups and  finitely many conjugacy classes of subgroups of order $p^n$ for $n\geq0$ (see \cite[Lemma 1.4]{BLO}).

We will use these properties of discrete $p$-toral groups and the description of the restriction maps in the case of finite $p$-groups to determine $T(P)$ in terms of the $P$--conjugacy classes of maximal elementary abelian subgroups of $P$ of rank 2. The result will be analogous to the case of finite $p$-groups; for almost all the cases the group of invertible modules is an abelian free group. We shall deal separately with  discrete $p$-toral groups that admit a tower of cyclic, dihedral, semidihedral or quaternion groups. We will describe these cases first.

\begin{proposition}\label{Prop 5.7}
The following hold. 
\begin{itemize}
    \item[(a)] Let $P=\mathbb{Z}/p^\infty$. Then $T(P)\cong \mathbb{Z}/2$. 
    \item[(b)] Let $D_{2^{\infty}}=\bigcup D_{2^n}$, where $D_{2^n}$ denotes the dihedral group of order $2^n$, for $n\geq 3$. Then $T(D_{2^\infty})\cong \mathbb{Z}$.
    \item[(c)] Let $Q_{2^\infty}=\bigcup Q_{2^n}$, where $Q_{2^n}$ denotes the generalized quaternion group of order $2^n$, for $n\geq 3$. Then $T(Q_{2^\infty})\cong \mathbb{Z}/4$.  
\end{itemize}

\end{proposition}

\begin{proof}

    $(a)$ In this case  $P$ admits a tower of cyclic groups $$\mathbb{Z}/p\leq \mathbb{Z}/{p^2}\leq...$$ Since $T(\mathbb{Z}/{2})$ is trivial and $T(\mathbb{Z}/{p^m})\cong \mathbb{Z}/2$ generated by $[\Omega(k)]$, for $p^m>2$ (see \cite[Corollary 8.8]{Dad78b}), we deduce that the restriction map $\mathrm{Res}\colon T(\mathbb{Z}/{p^{m+1}})\to T(\mathbb{Z}/{p^{m}})$ is the identity, for all $m>1$. Hence the  inverse system $(T(\mathbb{Z}/{p^m}),\mathrm{Res})$ is eventually constant. Then $T(P)$ is isomorphic to $\mathbb{Z}/2$. \\
    
\noindent    $(b)$ Fix the following presentation for the dihedral group $D_{2^n}$. $$\langle r,s\mid r^{2^{n-1}}=s^2=(sr)^2=1\rangle$$ Then we will consider $D_{2^{n-1}}$ as the subgroup of $D_{2^n}$ generated by $r^2$ and $s$. 
    
    Recall that $T(D_{2^n})=\langle \Omega_{D_{2^n}},[L]\rangle\cong\mathbb{Z}^2$ (see \cite[Section 5]{CT00}).  Let $\Omega_{D_{2^{n-1}}},[L']$ be the standard generators of $T(D_{2^{n-1}})$. It is clear that $$\mathrm{Res}([\Omega_{D_{2^n}}])=[\Omega_{D_{2^{n-1}}}].$$ Set $F=\langle r^2s, r^{2^{n-2}}\rangle$ and   $F'=\langle s, r^{2^{n-2}}\rangle$, which are  representatives of the two conjugacy classes of maximal  elementary abelian subgroups of $D_{2^{n-1}}$. Since $F$ and $F'$ are conjugate in $D_{2^{n}}$, we have 
    \[
\begin{split}
 \mathrm{Res}\colon T(D_{2^n})\to & T(F)\oplus T(F') \\
[L]\mapsto & (-\Omega_F,-\Omega_{F'})
\end{split}
\]
    
\noindent (see \cite[Theorem 5.4]{CT00}). The map $\mathrm{Res}\colon T(D_{2^n})\to T(F)\oplus T(F')$ factors through $T(D_{2^{n-1}})$  so we have a commutative triangle

\centerline{
\xymatrix{  T(D_{2^n}) \ar[r] \ar[rd] & T(D_{2^{n-1}}) \ar[d]  \\ & T(F)\oplus T(F')   }}

\noindent and by the detection theorem (see \cite[Conjecture 2.6]{CT00}), the vertical map is injective. We deduce that   
\[
\begin{split}
 \mathrm{Res}\colon T(D_{2^{n}}) \to & T(D_{2^{n-1}}) \\
m\Omega_{D_{2^n}}+n[L]\mapsto & (m-n)\Omega_{D_{2^{n-1}}}
\end{split}
\]
  Hence the result follows. \\

\noindent $(c)$ Fix the following presentation for the generalized quaternion group $Q_{2^{n+1}}$ $$\langle x,y\mid x^{2^n}=1,y^2=x^{2^{n-1}}, yxy=x^{-1}\rangle$$ We will identify $Q_{2^{n}}$ with the subgroup $\langle x^2,y\rangle$ of $Q_{2^{n+1}}$. 

Recall that $T(Q_{2^{n+1}})=\langle\Omega_{Q_{2^{n+1}}},[\Omega^1_{Q^{n+1}}(L)]\rangle\cong\mathbb{Z}/4\oplus\mathbb{Z}/2$ (see \cite[Section 6]{CT00}). Consider the restriction map $$\mathrm{Res}\colon T(Q_{2^{n+1}})\to T(H)\oplus T(H')$$ where $H=\langle x^{2^{n-2}},y\rangle $ and $H'=\langle x^{2^{n-2}},xy\rangle$ are representatives of the two conjugacy classes of quaternion subgroups of order 8. Then  $$\Omega^1_{Q_{2^{n+1}}}(L)\mapsto(2\Omega_H,0)\textrm{ or }(0,2\Omega_{H'})$$ under the previous restriction map \cite[Theorem 6.5]{CT00}.  Let $F=\langle (x^2)^{2^{n-3}}, y\rangle$ and   $F'=\langle (x^2)^{2^{n-3}}, x^2y\rangle$ be representatives of the two conjugacy classes of quaternion subgroups of $Q_{2^{n}}$ of order 8. Note that $F$ and $F'$ are $Q_{2^{n+1}}$-conjugate, hence 
\[
\begin{split}
 \mathrm{Res}\colon T(Q_{2^{n+1}})\to & T(H)\oplus T(H') \\
\Omega_{Q_{2^{n+1}}}\mapsto &(\Omega_F,\Omega_{F'})\\
[\Omega_{Q_{2^{n+1}}}^1(L)]\mapsto & (0,0) \textrm{ or } (2\Omega_F,2\Omega_{F'})
\end{split}
\]
 Since the map $\mathrm{Res}\colon T(Q_{2^{n+1}})\to T(F)\oplus T(F')$ factors through $T(Q_{2^n})$ we have a commutative diagram

\centerline{
\xymatrix{    T(Q_{2^{n+1}}) \ar[r] \ar[rd] & T(Q_{2^{n}}) \ar[d]    \\ & T(H)\oplus T(H')   }}

\noindent and by the detection theorem, the vertical arrow is injective. We deduce that
\[
\begin{split}
 \mathrm{Res}\colon T(Q_{2^{n+1}})\to & T(Q_{2^n}) \\
(m\Omega_{Q_{2^{n+1}}},n[\Omega^1_{Q^{n+1}}(L)])\mapsto & ((n+2m)\Omega_{Q^{2^n}},0) \textrm{ or } \\ &  (m\Omega_{Q^{2^n}},0)
\end{split}
\]

\noindent Therefore $\displaystyle \varprojlim T(Q_{2^n})\cong\mathbb{Z}/4$.
\end{proof}

\begin{remark}
The maximal subgroups of a  semi-dihedral group are generalized quaternion, dihedral and cyclic groups and none of them contain a semi-dihedral group as a subgroup (see \cite[Theorem 4.3]{Gor80}).  Then a locally finite group that admits a tower of semi-dihedral groups is a semi-dihedral group, hence a finite group.
 
\end{remark}

\begin{proposition}\label{Prop 5.9}
Let $P$ be a discrete $p$-toral group that admits a tower  $Q_1\leq Q_2\leq ...$ such that $Q_n$ is not  cyclic, dihedral, semi-dihedral or quaternion for all $n\geq1$. Then
$$
T(P) = \left\{
        \begin{array}{ll}
           
            \mathbb{Z}^r & \textrm{ if } P \textrm{ has }p\textrm{-rank  at most 2}\\
            \mathbb{Z}^{r+1} & \textrm{ if } P \textrm{ has }p\textrm{-rank  at least 3}
        \end{array}
    \right.
$$
  
\noindent where $r$ is the number of conjugacy classes of maximal elementary abelian subgroups of $P$ of rank 2.  
\end{proposition}

\begin{proof}

Since $Q_n$ is not cyclic, semi-dihedral or generalized quaternion for all $n\geq1$, we have that  $T(Q_n)$ is a free abelian group and its rank is determined by the connected components of $\mathcal{E}_{\geq2}(Q_n)/Q_n$, the poset of $Q_n$--orbits of elementary abelian subgroups of $p$-rank at least 2.  Recall that we have a finite number of $P$--conjugacy classes of maximal elementary abelian subgroups of rank 2. Fix  representatives  $E_1,...,E_r$ of these classes. We can assume  that $E_i$ is a subgroup of $Q_1$ and has the form $E_i=\langle u_i\rangle\times Z$, where $Z$ is the unique central subgroup of $Q_1$ of order $p$ and $\langle u_i\rangle$ is a non-central subgroup  of $Q_1$ of order $p$, for $1\leq i\leq r$ (see \cite[Section 3.3]{Maz}).  

Suppose that $P$ has $p$-rank at least 3. Hence we can assume that $Q_1$ has $p$-rank at least 3. For $n\geq1$,  choose elementary abelian subgroups $E_0^n,...,E_{r+s_n}^n$    of rank 2 which are representatives of the connected components of $\mathcal{E}_{\geq2}(Q_n)/Q_n$. We can assume that $E^n_0=E_0$ for a fixed elementary abelian subgroup in the big component of $\mathcal{E}_{\geq2}(Q_1)/Q_1$, that is, the connected component that contains all the elementary abelian subgroups of $Q_1$ of rank at least 3. 

We can assume that $E_i^n=E_i$ for $1\leq i\leq r$, and  $E_i^n=Z\times\langle u^n_i\rangle$ for a non-central subgroup $\langle u^n_i\rangle$ of $Q_n$ of order $p$, for $r+1\leq i\leq r+s_n$. Then there exist endotrivial modules $N_1^n,..,N_{r+s_n}^n$ such that
$$\mathrm{Res}_{E^n_j}^{Q_n}(N^n_i)\cong\left\{
        \begin{array}{ll}
           
            k\oplus(\textrm{proj}) & \textrm{ if } i\not=j,\\[7pt]
            \Omega_{E_j^n}^{-2p}(k)\oplus (\textrm{proj}) & \textrm{ if } i=j \textrm{ and } C_{Q_n}(u_i^n)/\langle u_i^n\rangle \textrm{ is cyclic of order at least 3,}\\[7pt]
        \Omega_{E_j^n}^{-2}(k)\oplus (\textrm{proj}) & \textrm{ if } i=j \textrm{ and } C_{Q_n}(u_i^n)/\langle u_i^n\rangle \textrm{ has order 2,}\\[7
        pt]
   \Omega_{E_j^n}^{-8}(k)\oplus (\textrm{proj}) & \textrm{ if } i=j \textrm{ and } C_{Q_n}(u_i^n)/\langle u_i^n\rangle \textrm{ is quaternion.}
        \end{array}
    \right.$$
\noindent for $0\leq j \leq r+s_n$ and $1\leq i \leq r+s_n$ (see \cite[Section 3.3]{Maz}).

Since we have only $r$ classes of $P$--conjugation of maximal elementary abelian subgroups or rank 2, the subgroup $E_i^n$ must be in the same $Q_m$--orbit of some $E_j$ in $\mathcal{E}_{\geq2}(Q_m)/Q_m$ for some $m\geq n$, and some $j=0,...,r$. We can suppose  that this holds for $m=n+1$. In particular, we have that $$\mathrm{Res}^{Q_{n+1}}_{E^n_j}([N^{n+1}_i])=[k]$$ for $r< i\leq r+ s_{m+1}$ and $0\leq j \leq r+ s_m$. Then it follows that $\mathrm{Res}^{Q_{n+1}}_{Q_n}$ is trivial on the $N^n_j$-components for $r< j\leq r+s_n$. On the other hand, note that we can find a large enough $k$  such that $$C_{Q_n}(u_i^n)/\langle u_i^n\rangle\cong C_{Q_{n+1}}(u_i^{n+1})/\langle u_i^{n+1}\rangle$$ for all $n\geq k$ since $C_{Q_{n}}(u_i)\leq C_{Q_{n+1}}(u_i)$. We can suppose that $k=1$. Then we have that
$\mathrm{Res}^{Q_{n+1}}_{Q_n}[N^{n+1}_j]=[N^n_j]$ for $0\leq j\leq r$.

 For $j\geq1$, define $\pi_j\colon\mathbb{Z}^{r+1}\to T(Q_j)$ as the inclusion on the generators $N^n_j$ for $0\leq j\leq r$. It is straightforward to show that $(\mathbb{Z}^{r+1},\pi_i)$ is the limit of the inverse system $\{ T(Q_n) \}$.  Then the result holds.  The case where $P$ has $p$-rank at most 2 is analogous. 
\end{proof}

The following result is analogous to the description of finite $p$-groups whose group of endotrivial modules is infinite cyclic. For abelian $p$-groups this  is precisely the main theorem given by Dade in \cite{Dad78b}. See also \cite[Theorem 3.5]{Maz}.

\begin{corollary}\label{10} Let $P$ be a discrete $p$-toral group. If one of the following conditions holds, then $T(P)\cong \mathbb{Z}$.

\begin{itemize}
\item[$(1)$] $P$ is an abelian group of $p$-rank at least 2.

\item[$(2)$] $P$ has $p$-rank at least $p+1$ if $p$ is odd or at least 5 if $p=2$.
\end{itemize}

\end{corollary}

\begin{proof}
If $(1)$ holds, then $P$ admits a tower of $p$-abelian groups of rank at least 2. The group of endotrivial modules of such groups is infinite cyclic by \cite[Theorem 10.1]{Dad78b}, and we deduce that  the restriction maps are all the identity. If $(2)$ holds, the result follows in a similar fashion  by  \cite[Theorem 3.5]{Maz}. 
\end{proof}

Let $P$ be a discrete $p$-toral group. Let  $P_1\leq P_2\leq... $ be a tower for $P$. If $P_n$   is cyclic, dihedral, semi-dihedral or quaternion for just a finite number of $n$, then we can always consider ignore the first few terms and re-index the tower  so that $P$ satisfies the hypothesis of Proposition \ref{Prop 5.9}. If $P_n$ is cyclic, dihedral, semi-dihedral or quaternion for an infinite number of $n$, then we can extract a tower so that all the terms are of the same type, hence $P$ would be isomorphic to one of the groups in Proposition \ref{Prop 5.7}. Thus we have covered completely the class of artinian countable locally finite $p$-groups. 

The following result completes the classification of the Picard group for the class of countable locally finite $p$-groups. 

\begin{proposition}
Let $P$ be a countable locally finite $p$-group. If $P$ is not artinian, then $T(P)\cong\mathbb{Z}$.
\end{proposition}

\begin{proof}
By Lemma 3.1 in \cite{KW}, we have that $P$ contains an infinite elementary abelian subgroup. Hence there is a tower $P_1\leq P_2\leq P_3\leq ...$ so that $P_n$ has $p$-rank at least $p+4$, for all $n\geq1$. By  \cite[Theorem 3.5]{Maz} we have $T(P_n)\cong \mathbb{Z}$, for all $n\geq1$. We deduce that  $\{T(P_n)\}$ is constant, and the result follows.  
\end{proof}

\subsection{Countable locally finite groups}

\begin{remark}
For general finite groups, the description of the group of endotrivial modules by generators and relations is not complete. Hence it is more elaborate to identify the restriction maps $\mathrm{Res}\colon T(G_{n+1})\to T(G_n)$ in an inverse system for a countable locally finite group $G$. A different approach is to study the restriction map  $\mathrm{Res}\colon T(G)\to T(S)$ where $S$ is a maximal $p$-subgroup of $G$. However, we need to be careful since there are locally finite groups with maximal $p$-subgroups which are not necessarily  isomorphic (see \cite[Section 1.D]{KW} for a discussion). 
\end{remark}

\begin{definition}
Let $G$ be a group. We say that $G$ is \textit{$p$-artinian} if any $p$-subgroup of $G$ is artinian. 
\end{definition}

Let $G$ be a $p$-artinian countable locally finite group. In this case, a maximal $p$-subgroup of $G$ plays the role of a Sylow $p$-subgroup. In particular, any two maximal $p$-subgroups of $G$ are isomorphic (see \cite[Theorem 3.7]{KW}).  Consider a tower of finite groups $G_1\leq G_2 \leq ...$ of $G$.  Assume that $p$ divides the order of $G_1$.  Set $S_1$ a $p$-Sylow subgroup of $G_1$.  For each $n\geq2$ we can find a $p$-Sylow subgroup $S_n$ of $G_n$ such that $S_{n-1}\leq S_n$. Then we obtain a ascending chain $S_n$ of finite $p$-subgroups of $G$. Then $S=\cup S_n$ is a maximal $p$-subgroup. Moreover, note that $S$ is a discrete $p$-toral group.

By Proposition \ref{prop toral} we know that $T(G)$ agrees with the projective limit of the inverse system $\{T(G_n)\}$ with maps induced by the restriction maps. We would like to give a better description of this projective limit.

For $n\geq1$ consider the restriction map $\mathrm{Res}\colon T(G_n)\to T(S_n)$ and denote its image by $\overline{T}(S_n)$  and its kernel by $T(G_n,S_n)$. Then we have a short exact sequence of abelian  groups $$0\to T(G_n,S_n)\to T(G_n)\to  \overline{T}(S_n)\to 0 .$$ 

Then we obtain a short exact sequence of inverse systems. The following sequence is exact since $T(G_n,S_n)$ is finite for all $n\geq1$, hence we can use the Mittag-Leffler condition  for the vanishing of the ${\displaystyle \varprojlim}  ^1$.  
$$0\to \displaystyle \varprojlim_{n} T(G_n,S_n)\to T(G)\to \displaystyle \varprojlim_{n} \overline{T}(S_n)\to 0.$$

We can identify $\varprojlim_{n} \overline{T}(S_n)$ as a subgroup of $T(S)$. Moreover, this group agrees with the image of the restriction map $\mathrm{Res}\colon T(G)\to T(S)$ and as a consequence we obtain that $\mathrm{Ker}(\mathrm{Res})\cong \varprojlim_{n} T(G_n,S_n)$. We will denote this group by $T(G,S)$. 

As we mentioned before, we have the inconvenience that we do not have an explicit description of the restriction maps for arbitrary finite groups as in the case of finite $p$-groups. Thus the best we can do is considering special cases of countable locally finite groups. We will conclude this section with a couple of examples where we are able to determine $T(G)$ as an abstract group.  

\begin{example}
Let $G$ be an abelian countable locally finite group that has $p$-rank at least 2. Suppose additionally that $G$ is $p$-artinian. Let $S$ be a maximal $p$-subgroup of $G$ constructed from a tower $\{S_n\}$ of Sylow $p$-subgroups as above. Since $G$ is of $p$-rank at least 2, we can prove that $T(S)$ is infinite cyclic.  Hence, we deduce that $\varprojlim \overline{T}(S_n)\cong T(S)$. On the other hand, we have that $T(G_n,S_n)\cong \Hom(G_n,k^\times)$. It follows that  $$T(G)\cong \mathbb{Z}\oplus \Hom(G,k^\times).$$

\end{example}

\begin{example}
 We say that a locally finite group $G$ is \textit{$p$-nilpotent} if any $p$-subgroup of $G$ is nilpotent.  Let $G$ be a locally finite group that is $p$-artinian and $p$-nilpotent. Then we have a tower $G_1\leq G_2\leq ...$ where $G_n$ is a finite $p$-nilpotent group for any $n\geq1$. By \cite[Theorem 3.3]{CMT11} we have a short exact sequence of abelian groups $$0\to\Hom(G_n,k^\times)\to T(G_n)\to T(S_n)\to 0$$ where $S_n$ is a Sylow $p$-subgroup of $G_n$, for $n\geq1$. Then  $$T(G)\cong \Hom(G,k^\times)\oplus T(S)$$ where $S$ is a maximal $p$-subgroup of $G$. 
\end{example}


\subsection{Amalgam groups}\label{Subsection amalgam groups}

In this subsection, we let $G$ be a group with geometric dimension two for the family of finite groups. In particular, we are interested in the case where the fundamental domain of the action is homeomorphic to the standard 2-simplex. 

\begin{definition}
    We will say that $G$ is an \textit{amalgam group}, if it admits a 2-dimensional model $X$ for $\underline{E}G$ such that the fundamental domain of the action is homeomorphic to the standard 2-simplex. In particular, amalgam groups are groups of type $\Phi$. 
\end{definition}

Let $\mathcal{T}$ denote the barycentric subdivision of $\Delta^2$. A triangle of groups is a functor $\mathcal{T}\to \textrm{Gps}$ that can be depicted as commutative diagram of groups

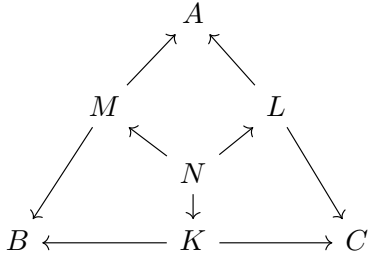
\begin{center}

\begin{tikzcd}[column sep=small, row sep= small] & & A & & \\ & & & & \\   & M \arrow[ldd] \arrow[ruu] & &   L \arrow[luu] \arrow[rdd] & \\ & & N \arrow[lu]\arrow[ru]\arrow[d] & & \\ B & & K \arrow[ll] \arrow[rr] & & C \end{tikzcd}

\end{center}

\noindent where all the maps are injective maps.  The groups $A,B$ and $C$ are called vertex groups, $K,L$ and $M$ are called edge groups and $N$ is called the face group. In other words, it corresponds to a functor from the barycentric subdivision of the standard 2-simplex that associates a group to each vertex, an injective morphism to each edge, and a compatibility data for the composition on each face (see  \cite[Section 1]{FP}). 

Any amalgam group defines a triangle of groups where the groups correspond to the isotropy groups of the vertices, edges and face, and the maps are given by inclusions. However, this assignation is not bijective, there are triangles of groups that do not correspond to an amalgam group.  If the triangle of groups is non-positively curved, then the fundamental group of the triangle of groups is an amalgam group (see \cite{Hae}, \cite{Sta}).

 For an amalgam group $G$, consider the associated diagram of $\infty$-categories $\StMod (-)\colon  \mathcal{T}^\textrm{op}\to \widehat{\textrm{Cat}}^\otimes_\infty$ obtained by composing the stable module $\infty$-category functor, and the corresponding triangle of groups $\mathcal{T}\to \mathrm{Gps}$ associated to $G$. In particular, the value of this functor at an element $\sigma$ of $\mathcal{T}$ is precisely $\StMod(kG_\sigma)$.   Recall that the elements in $\mathcal{T}$ are simplices in $\Delta^n$, so here $G_\sigma$ means the isotropy group of the simplex of $X$ corresponding to the simplex $\sigma$ of $\Delta^2$. 

\begin{proposition}\label{decomposition for EG}
  Let $G$ be an amalgam group with trivial face group. Then there is an equivalence of symmetric monoidal $\infty$-categories $$ \StMod(kG)\xrightarrow[]{\simeq} \varprojlim_{\sigma\in \mathcal{T}^\textrm{op}} \StMod(kG_\sigma). $$ 
\end{proposition}

\begin{proof}
Consider the fundamental group $\Tilde{G}$ of the graph of groups obtained from the action of $G$ on the 1-skeleton of $X$. In particular, this graph of groups is indexed by the barycentric subdivision of $\partial \Delta^2$ which, for simplicity, we will denote it  by $\mathcal{T}'$.  In this case, \cite[Theorem I.9.2]{DD89} gives us an extension of groups $$1\to N\to \Tilde{G} \xrightarrow[]{\pi} G\to 1$$ where $N$ is a free group. Hence any finite subgroup $H$ of $\Tilde{G}$ has trivial intersection with $N$, so it can be identified with a finite subgroup of $G$ under $\pi$, and any finite subgroup of $G$ can be lifted to a finite subgroup of $\Tilde{G}$. Now, we claim that the functor $\mathcal{O}_\mathscr{F}(\Tilde{G})\to \mathcal{O}_\mathscr{F}(G)$ induced by $\pi$ is cofinal. Indeed, the same argument as in Proposition \ref{cofinalidad de la inclusion de subgrupos de isotropia} works. Let $G/H$ be an element in $\mathcal{O}_\mathscr{F}(G)$. Any element $G/H\xrightarrow[]{[g]} G/K$ in $\mathcal{O}_\mathscr{F}(\Tilde{G})_{(G/H) / }$ is isomorphic to $G/H \xrightarrow[]{[1]} G/gKg^{-1}$. Hence we can identify $\mathcal{O}_\mathscr{F}(\Tilde{G})_{(G/H) / }$ with the opposite of the poset of subgroups in $\mathscr{F}$ that contain $H$, which is weakly contractible since $H$ itself is a minimum of this poset. Then Quillen's Theorem A completes the claim. Moreover, by virtue of Proposition \ref{decomposition of C} we obtain equivalences of symmetric monoidal $\infty$-categories $$\StMod(kG)\simeq \varprojlim_{G/H\in \mathcal{O}_{\mathscr{F}}(G)^\textrm{op}} \StMod(kH)\simeq \varprojlim_{G/H\in \mathcal{O}_{\mathscr{F}}(\Tilde{G})^\textrm{op}} \StMod(kH).$$

On the other hand,  Theorem \ref{graphs of groups} and Proposition \ref{decomposition of C} give us equivalences of symmetric monoidal $\infty$-categories $$\varprojlim_{\sigma\in \mathcal{T'}^\textrm{op}} \StMod(kG_\sigma)\simeq \StMod(k\Tilde{G})\simeq \varprojlim_{G/H\in \mathcal{O}_{\mathscr{F}}(\Tilde{G})^\textrm{op}} \StMod(kH). $$ But note that the functor $\StMod (-)\colon\mathcal{T}^\textrm{op}\to \widehat{\textrm{Cat}}^\otimes_\infty$ is the right Kan extension of the functor $\StMod (-)\colon\mathcal{T'}^\textrm{op}\to \widehat{\textrm{Cat}}^\otimes_\infty$ along the inclusion $\mathcal{T'}^\textrm{op} \to \mathcal{T}^\textrm{op}$ since the stable module $\infty$-category of the trivial group $\StMod(k\{1\})$ is trivial. Hence we obtain an equivalence of symmetric monoidal $\infty$-categories $$\varprojlim_{\sigma\in \mathcal{T'}^\textrm{op}} \StMod(kG_\sigma)\simeq \varprojlim_{\sigma\in \mathcal{T}^\textrm{op}} \StMod(kG_\sigma)$$
which gives us an equivalence with $\StMod(kG)$ as we wanted. 
\end{proof}

\begin{corollary}\label{exact sequence for triangles of groups}
    Let $G$ be an amalgam group with trivial face group and $\mathcal{T}\to \textrm{Gps}$ be its associated triangle of groups. Then we have an exact sequence of abelian groups $$0\to H^1(\mathcal{T}; \pi_1\circ f)\to T(G)\to H^0(\mathcal{T}; \pi_0\circ f)\to0$$ where $f$ is the composition of the Picard space functor and the stable module $\infty$-category functor  corresponding to the triangle of groups $\mathcal{T}\to \textrm{Gps}$.
\end{corollary}

\begin{proof}
This follows in the same fashion as in Corollary \ref{exact sequence for picard groups}.  Consider the spectral sequence of Equation (\ref{spectral sequence for Pic}) $$E_2^{p,q}=H^p(\mathcal{T};\pi_q\circ\mathrm{Pic}\circ \StMod)\Rightarrow \pi_{p-q}(\mathrm{Pic}(\StMod(kG))).$$ Since the face group is trivial,  $E^{p,q}$ is trivial except possibly for $p=0,1$. 
\end{proof}

\begin{example}
    Consider the Coxeter group $G=\Delta^*(2,4,4)$ of isometries of the Euclidean plane generated by the reflections across the sides of a triangle with angles $\pi/2,\pi/4$ and $\pi/4$.  In this case, $G$ is an amalgam group arising from a non-positively curved triangle of groups with trivial face group, with each edge group isomorphic to $\mathbb{Z}/2$ and with vertex groups isomorphic to the dihedral groups $D_{8}, D_{16}$ and $ D_{16}$, where $D_{n}$ denotes the dihedral group of order $n$.

    In particular, the diagram  $\pi_1\circ\mathrm{Pic}\circ\StMod$ is constant. It follows that $H^1(\mathcal{T};\pi_1\circ\mathrm{Pic}\circ\StMod)\cong H^1(|\mathcal{T}|;k^\times)\cong  k^\times$. Since $T(\mathbb{Z}/2)=0$, we obtain 
    $$T(G)\cong  T(D_8)\oplus T(D_{16})\oplus T(D_{16}) \oplus k^\times \cong \mathbb{Z}^6 \oplus k^\times.$$
\end{example}

\bibliographystyle{amsplain}
\bibliography{mybibfile}

\end{document}